\documentclass[12pt,a4paper]{amsart}

\usepackage[top=35mm,bottom=35mm,left=30mm,right=30mm]{geometry}
\usepackage{amssymb}
\usepackage{enumitem}
\usepackage{mathptmx}
\usepackage{mathrsfs}
\usepackage{xcolor}
\usepackage[colorlinks=true,linkcolor=blue,
citecolor=blue,urlcolor=blue]{hyperref}

\newtheorem{theorem}{Theorem}[section]
\newtheorem{proposition}[theorem]{Proposition}
\newtheorem{lemma}[theorem]{Lemma}
\newtheorem{corollary}[theorem]{Corollary}
\newtheorem{question}{Question}
\newtheorem{condition}{Condition}
\newtheorem*{claim}{Claim}

\newtheorem*{step1}{Step 1}
\newtheorem*{step2}{Step 2}
\newtheorem*{step3}{Step 3}

\theoremstyle{definition}
\newtheorem{definition}[theorem]{Definition}
\newtheorem{fact}[theorem]{Fact}
\newtheorem{example}[theorem]{Example}

\theoremstyle{remark}
\newtheorem{remark}[theorem]{Remark}

\newcommand{\scl}{\mathrm{scl}}

\numberwithin{equation}{section}

\begin{document}

	\title{Entropy Scales in Topological Dynamical Systems}

	\author[Yunxiang Xie, Ercai Chen and Xiaoyao Zhou]
	{Yunxiang Xie, Ercai Chen and Xiaoyao Zhou\textsuperscript{*}}

	\thanks{\textsuperscript{*}Corresponding author.}

	\address{School of Mathematical Sciences, Ministry of Education Key Laboratory of
		NSLSCS, Nanjing Normal University, Nanjing 210023, Jiangsu, P.R. China}

	\email{yxxie20@126.com}
	\email{ecchen@njnu.edu.cn}
	\email{zhouxiaoyaodeyouxian@126.com}


	\begin{abstract}
In this paper, we  define topological and  measure-theoretic entropy scales in dynamical systems for the scaling introduced by Helfter.
Firstly, we establish a variational principle between Bowen and lower entropy scales on compact subsets. Secondly, we prove factor inequalities for Bowen and upper capacity entropy scales on arbitrary subsets; unlike the classical additive case, the fiber contribution is described by a transfer function and need not be additive. We also study entropy scales and sequence entropy scales under the induced action. Under suitable conditions, zero values are preserved in both directions, while positive values become infinite; counterexamples show that additional assumptions are generally necessary. As an application, we use entropy scales to distinguish zero-entropy systems sharing the same entropy dimension but differ in finer orbit-growth rates.

	\end{abstract}


	\maketitle
	\section{Introduction and main results}

	Let $(X,T)$ be a topological dynamical system (TDS for short),
	where $(X,d)$ is a compact metric space and $T:X\to X$ is
	continuous. Denote by $\mathcal D(X)$ the set of compatible
	metrics on $X$, and by $\mathscr M(X)$ and $\mathscr M_T(X)$
	the sets of Borel probability measures and $T$-invariant Borel
	probability measures on $X$, respectively.

	Measure-theoretic entropy and topological entropy are fundamental invariants in classical ergodic theory, characterizing the complexity and chaotic behavior of dynamical systems. The former was introduced by Kolmogorov \cite{Kol58} and Sinai \cite{Sin59} for measure-preserving systems, while the latter, topological entropy, was later formulated by Adler, Konheim, and McAndrew \cite{AKM65}. A variational principle, due to Goodwyn, Dinaburg, and Goodman \cite{Din70,Goo69,Goo71}, relates these two notions and is of foundational importance in both ergodic theory and dynamical systems:
	\begin{align*}
		h_{top}(T,X)
		=
		\sup_{\mu\in\mathscr M_T(X)}h_\mu(T).
	\end{align*}
Bowen
\cite{Bow73} introduced topological entropy for arbitrary
subsets through a construction analogous to Hausdorff
dimension. Later, Feng and Huang \cite{FH12} proved that, for every
nonempty compact subset $K\subset X$,
\begin{align}\label{equ 1.1}
	h_{top}^{B}(T,K)
	=
	\sup\left\{
	\underline h_\mu(T):
	\mu\in\mathscr M(X),\ \mu(K)=1
	\right\},
\end{align}
where $\underline h_\mu(T)$ denotes the lower
Brin--Katok local entropy of $\mu$. 	Entropy theory has since developed close connections with
	topological pressure, dimension theory and the study of
	invariant measures \cite{Pes97,PP84,Wal82}.

	Classical topological entropy records the exponential growth rate of distinguishable orbit segments, and its value distinguishes systems with zero, positive finite, or infinite entropy. Within the zero entropy and infinite entropy classes, however, this single invariant fails to detect finer levels of complexity. This limitation has prompted several refinements. Mean dimension, introduced by Gromov \cite{Gro99}, and metric mean
	dimension, developed by Lindenstrauss and Weiss \cite{LW00},
	provide further information about systems with infinite
	topological entropy. For zero-entropy systems, slow entropy measures finer orbit growth \cite{Fer97,KT89}, while entropy dimension describes intermediate growth between polynomial and exponential orders \cite{DHP11}. More general time-growth regimes have also been investigated. Galatolo \cite{Gal03} introduced generalized topological entropy by normalizing the growth of separated and spanning sets with a prescribed function, with particular emphasis on weakly chaotic and zero-entropy systems; in subsequent work \cite{Gal07}, he related generalized local entropy to covering complexity.  Zhao and Pesin \cite{ZP15,ZP16} introduced scaled
	topological and metric entropies by using weights
	$e^{-\alpha a(n)}$, where $a(n)$ is a prescribed increasing sequence. These approaches underscore the importance of selecting an appropriate growth scale when classical entropy provides insufficient information.

	A similar issue arises in dimension theory when classical
	dimensions take the value zero or infinity. Helfter
	\cite{Hel25a} introduced the definition of \ \emph{scaling}, a one-parameter
	family $\scl=\{\scl_\alpha\}_{\alpha\geq0}$ of gauge functions,
	and defined the associated Hausdorff, packing, box, local and
	quantization scales. This framework compares different types of
	dimension and allows a finer study of
	infinite-dimensional spaces and measures. Its applications
	include the geometry of ergodic decompositions, Wiener measure
	and function spaces. Helfter \cite{Hel25b} subsequently studied
	the realization of prescribed Hausdorff and packing scales by
	compact subsets of infinite-dimensional Banach spaces.
	Fan and Helfter \cite{FH25} developed multifractal analysis
	for Wiener measure using scales and established a generalized
	Frostman lemma on a class of Polish spaces.

	Motivated by these developments, it is natural to consider whether
	one can define entropy quantities in a more general framework
	to study and distinguish finer levels of complexity in dynamical
	systems, particularly those with zero  topological entropy.
	In this paper, for a given scaling, we define several entropy scales in both topological and measure-theoretic senses. The new notion of entropy scale include several entropy quantities as special cases, including \cite{DHP11,Fer97,Gal03,KT89,ZP15}.
	We are interested in whether variational principles hold for such entropy scales.

	\begin{question}\label{ques 1}
	 Does a variational principle of the form \eqref{equ 1.1} hold for any general scaling?
	\end{question}

We answer Question~\ref{ques 1} affirmatively for compact subsets.
The scaling axioms imply a summability property between
different parameter levels. This property allows us to compare
weighted and ordinary Bowen covers and establish the following
variational principle.

	\begin{theorem}\label{thm 1.1}
		Let $(X,T)$ be a TDS and let $\scl=\{\scl_\alpha\}_{\alpha\geq0}$ be a scaling. For every
		nonempty compact subset $K\subset X$,
		\begin{align*}
h_{\scl}^B(T,K)=
			\sup
			\left\{
			\underline h_{\mu,\scl}(T):
			\mu\in\mathscr M(X),\mu(K)=1
			\right\},
		\end{align*}
		where $h_{\scl}^B(T,K)$ and $\underline h_{\mu,\scl}(T)$ denote the Bowen entropy scale and the lower
		measure-theoretic entropy scale, respectively.
	\end{theorem}

We next study entropy scales under factor maps and examine how
fiber complexity contributes to the entropy scales of an extension.	Let $\pi:(X,T)\to(Y,S)$ be a factor map.
	 Bowen
	\cite[Theorem~17]{Bow71} proved the classical factor inequality
	\begin{align*}
		h_{top}(S,Y)
		\leq h_{top}(T,X)
		\leq h_{top}(S,Y)+\sup_{y\in Y}
		h_{top}\left(T,\pi^{-1}(y)\right).
	\end{align*}
	This bounds the entropy of an extension by the entropy of its
	factor and the maximal fiber entropy.
	The relation between factor entropy and fiber complexity has
	been developed in several directions. Ledrappier and Walters
	\cite{LW77} established a relative variational principle
	relating the entropy of invariant lifts to entropy along
	fibers. Walters \cite{Wal86} developed the corresponding
	theory of relative pressure and compensation functions.
	Huang, Ye and Zhang \cite{HYZ10} obtained factor estimates
	for compact subsets. For Bowen entropy, Oprocha and Zhang
	\cite{OZ11} established a subset factor inequality and proved
	that
	\begin{align*}
		\sup_{y\in Y}
		h_{top}\left(T,\pi^{-1}(y)\right)
		=
		\sup_{y\in Y}
		h_{top}^{B}\left(T,\pi^{-1}(y)\right).
	\end{align*}
	An analogous Bowen entropy estimate was shown in
	\cite[Theorem~3.3]{FHYZ12}. These results allow the complexity
	of a subset to be compared with that of its image, even when
	the subset is not invariant.
	Lifting covers from the factor introduces an exponential term
	governed by $\sup_{y\in Y}
	h_{top}\left(T,\pi^{-1}(y)\right)$. Its contribution is additive for classical
	entropy, while for general entropy scales it depends on the
	relative decay rates at different scaling levels.

	To describe this effect, for any $u,b\in[0,+\infty)$, we define the transfer function
	\begin{align*}
	\Gamma_{\scl}(u,b)
	:=
	\inf\left\{
	\gamma>0:  	\exists\,\alpha>u,\ \tau>0\  \text{such that}  \limsup_{n\to\infty}
	e^{(b+\tau)n}
	\frac{\scl_\gamma(e^{-n})}
	{\scl_\alpha(e^{-n})}
	<+\infty    \right\}.
	\end{align*}
	We use the convention $\inf\emptyset=+\infty$, and set
	$\Gamma_{\scl}(u,b)=+\infty$ if $u=+\infty$ or $b=+\infty$.
	The transfer function $\Gamma_{\scl}$ records which scaling
	levels can control the additional exponential covering growth.
	Then we obtain the following result of factor map.

	\begin{theorem}
		\label{thm 1.2}
		Let $\pi:(X,T)\to(Y,S)$ be a factor map. Then for every nonempty subset $E\subset X$,
		\begin{align*}
			\overline h_{\scl}^{C}\left(S,\pi(E)\right)
			\leq
			\overline h_{\scl}^{C}(T,E)
			\leq
			\Gamma_{\scl}\left(
			\overline h_{\scl}^{C}\left(S,\pi(E)\right),\sup_{y\in Y}
			h_{top}\left(T,\pi^{-1}(y)\right)
			\right),\\
			h_{\scl}^{B}\left(S,\pi(E)\right)
			\leq
			h_{\scl}^{B}(T,E)
			\leq
			\Gamma_{\scl}\left(
			h_{\scl}^{B}\left(S,\pi(E)\right),\sup_{y\in Y}
			h_{top}\left(T,\pi^{-1}(y)\right)
			\right).
		\end{align*}
	\end{theorem}

We next study how entropy scales behave under the induced
action on the space of Borel probability measures. A classical result in the zero entropy setting concerns the behavior of topological entropy under the induced action.
Let $\mathscr{M}(X)$ be the space of all Borel probability measures on $X$
endowed with the weak$^*$  topology. Consider the   push-forward map $T_*:\mathscr M(X)\longrightarrow\mathscr M(X)$ with $T_*\mu=\mu\circ T^{-1}$.
Then $\mathscr M(X)$ is a compact
metrizable space and $T_*$ is continuous. The topological dynamical system $(\mathscr{M}(X), T_*)$  is called the
\emph{induced system} of $(X,T)$ on the space of probability measures.
 Bauer and Sigmund \cite{BS75} proved that positive
entropy of the original system implies infinite entropy of
the induced system, while Glasner and Weiss \cite{GW95}
established that zero entropy is preserved. Consequently,
\begin{align*}
	h_{top}(T,X)>0
	\quad\Longleftrightarrow\quad
	h_{top}\bigl(T_*,\mathscr M(X)\bigr)>0
	\quad\Longleftrightarrow\quad
	h_{top}\bigl(T_*,\mathscr M(X)\bigr)=+\infty.
\end{align*}
Related questions have been studied for other actions and
complexity invariants. Ji and Wang \cite{JW21}
obtained entropy dichotomies for free semigroup actions.
Burguet and Shi \cite{BS25} proved that positive topological
entropy of the original system is equivalent to positive,
and hence infinite, mean dimension of the induced system.
Shi and Zhang \cite{SZ25} extended this result to countably
infinite discrete amenable group actions.

This raises the question of whether the zero level equivalence and the positive to infinite dichotomy for entropy scales extend to arbitrary dynamical systems and their induced actions.
By employing empirical measures and Hoeffding's inequality, we establish a relation between the covering numbers of the original system and its induced space. For general systems, however, the classical zerolevel equivalence and positive to infinite implication need not persist and we will give  counterexamples  in Section~\ref{sec 5} to support the fact. This indicates that when the scalings are allowed to be arbitrary, a more refined classification and additional constraints on the scaling family become necessary. To this end, we introduce two conditions that control the case with   zero-entropy-scale and positive-entropy-scale regimes, respectively,

Denote by $\ell_\alpha(n)=-\log \scl_\alpha(e^{-n})$.
\begin{condition}\label{cond 2}
	There exists $\alpha_0>0$ such that, for every
	$0<\alpha<\alpha_0$, at least one of the following statements holds:
	\begin{enumerate}
		\item There exists $0<\beta<\alpha$ such that
		\begin{align*}
			\lim_{n\to\infty}
			\frac{\ell_\beta(n)\log(n+1)}{\ell_\alpha(n)}
			=0.
		\end{align*}

		\item One has
		\begin{align*}
			\inf_{\gamma>0}
			\limsup_{n\to\infty}
			\frac{\ell_\gamma(n)}{n}
			=0,
			\quad
			\liminf_{n\to\infty}
			\frac{\ell_\alpha(n)}{n}
			>0.
		\end{align*}
	\end{enumerate}
\end{condition}

For the positive-to-infinite implication, we use the fact
that every finite Cartesian power of the original system
embeds into the induced system. Denote by
$X^{(m)}=X\times\cdots\times X$ and
$T^{(m)}=T\times\cdots\times T$, the maximum product metric
gives $s_n\bigl(X^{(m)},\epsilon\bigr)
\geq s_n(X,\epsilon)^m.$
A suitable comparison between scaling levels allows this
power growth to produce arbitrarily large entropy scales.
We also need the following condition to achieve further control.

\begin{condition}\label{cond 3}
	For every $0<\alpha<\gamma$, at least one of the following statements holds:
	\begin{enumerate}
		\item $\limsup_{n\to\infty}
		\frac{\ell_\gamma(n)}{\ell_\alpha(n)}
		<+\infty.$

		\item $0<
		\liminf_{n\to\infty}
		\frac{\ell_\alpha(n)}{n}
		\leq
		\liminf_{n\to\infty}
		\frac{\ell_\gamma(n)}{n}
		<+\infty.$
	\end{enumerate}
\end{condition}

 We obtain the following
result for upper capacity entropy scales.

	\begin{theorem}
		\label{thm 1.3}
		Let $(X,T)$ be a TDS, and let $(\mathscr M(X),T_*)$ be the induced system on the space of Borel
		probability measures.
		\begin{enumerate}
			\item[\rm(1)] If the scaling satisfies Condition~\ref{cond 2}, then
			\begin{align*}
				\overline h_{\scl}^{C}(T,X)=0
				\quad\Longleftrightarrow\quad
				\overline h_{\scl}^{C}
				\left(T_*,\mathscr M(X)\right)=0.
			\end{align*}
			\item[\rm(2)] If the scaling satisfies Condition~\ref{cond 3}, then
			\begin{align*}
				\overline h_{\scl}^{C}(T,X)>0
				\quad\Longrightarrow\quad
				\overline h_{\scl}^{C}
				\left(T_*,\mathscr M(X)\right)=+\infty.
			\end{align*}

		\end{enumerate}
	\end{theorem}

The entropy quantities considered  measure orbit complexity along consecutive integer times. It is natural to consider how the measured complexity changes when a different sequence of observation times is chosen. Given an increasing sequence $\mathsf S={t_1<t_2<\cdots}$ of non-negative integers, one observes the system along $T^{t_1},T^{t_2},\ldots$. To quantify complexity along such sequences, Ku\v{s}nirenko \cite{Kus67} introduced measure-theoretic sequence entropy, which can distinguish dynamical behavior not detected by ordinary entropy. He proved that its vanishing along every increasing sequence characterizes discrete spectrum. Goodman \cite{Goo74} subsequently introduced the topological counterpart. Further developments include realization results for topological sequence entropy \cite{SYZ20} and relative sequence entropy for amenable group actions \cite{LWW25}.
The dependence on observation times also leads us to examine whether the entropy relations between a system and its induced system remain valid along a prescribed sequence. Qiao and Zhou \cite{QZ17} proved that, for a homeomorphism and any increasing sequence, the original system has zero topological sequence entropy if and only if the induced system does. We ask whether this equivalence persists when the classical exponential weights are replaced by a general scaling family.

We study sequence entropy scales and establish the following relations between the original system and its induced system.

	\begin{theorem}\label{thm 1.4}
		Let $(X,T)$ be a TDS, and let
		$\mathsf S=\{t_1<t_2<\cdots\}\subset\mathbb Z_+$
		be an increasing sequence.
		Then the following statements hold.
		\begin{enumerate}
			\item[\rm(1)]
			If $\scl$ satisfies
			Condition~\ref{cond 2}, then
			\begin{align*}
				\overline h_{\scl,\mathsf S}^{C}(T,X)=0
				\quad\Longleftrightarrow\quad
				\overline h_{\scl,\mathsf S}^{C}
				\left(T_*,\mathscr M(X)\right)=0.
			\end{align*}

			\item[\rm(2)]
			If $\limsup_{n\to\infty}
			\frac{\ell_\gamma(n)}{\ell_\alpha(n)}
			<+\infty$ for every $0<\alpha<\gamma,$
			then
			\begin{align*}
				\overline h_{\scl,\mathsf S}^{C}(T,X)>0
				\quad\Longrightarrow\quad
				\overline h_{\scl,\mathsf S}^{C}
				\left(T_*,\mathscr M(X)\right)=+\infty.
			\end{align*}
		\end{enumerate}
		Here $\overline h_{\scl,\mathsf S}^{C}(T,X)$ denotes
		the upper capacity sequence entropy scale along $\mathsf S$.
	\end{theorem}

We also give an application to orbit complexity within a fixed intermediate growth order. A suitable scaling family detects logarithmic corrections that are not distinguished by the leading growth exponent.

	The remainder of this paper is organized as follows. In Section~\ref{sec 2}, we introduce several
	entropy scales and establish a Billingsley-type theorem. In Section~\ref{sec 3}, we prove
	Theorem~\ref{thm 1.1}. In Section~\ref{sec 4}, we prove Theorem~\ref{thm 1.2} and compute the transfer
	function for several standard scaling families. In Section~\ref{sec 5}, we study the induced action
	on the space of probability measures and prove Theorems~\ref{thm 1.3} and~\ref{thm 1.4}. Finally,
	Section~\ref{sec 6} applies entropy scales to distinguish finer intermediate
	orbit-growth rates.

	\section{Preliminaries}
	\label{sec 2}
	In this section, we introduce two types of topological entropy scales and measure-theoretic entropy Then we establish the Billingsley-type theorem for entropy scales.

	\subsection{Scalings and topological entropy scales}\label{subsec 2.1}
	We first recall Helfter's definition of the scaling \cite{Hel25a}.
	\begin{definition}\label{def 2.1}
		A one-parameter family $\scl=(\scl_\alpha)_{\alpha\geq0}$ of positive nondecreasing functions on $(0,1)$
		is called a \emph{scaling} if for every $\alpha>\beta>0$ and every $\lambda>1$ sufficiently close
		to $1$, the following relations hold:
		\begin{align*}
			\scl_\alpha(\epsilon)
			=
			o\left(\scl_\beta(\epsilon^\lambda)\right),
			\quad
			\scl_\alpha(\epsilon)
			=
			o\left(\scl_\beta(\epsilon)^\lambda\right)
		\end{align*}
		as $\epsilon\to 0$.
	\end{definition}

	In the paper, we assume that  for every $\alpha>0$,
	\begin{align*}
		\lim_{\epsilon\to 0}\scl_\alpha(\epsilon)=0
	\end{align*}

	\begin{fact}\label{fact 2.2}
		Let $\scl$ be a scaling defined above. Then for any $\alpha>\beta>0$ and any constant $C>0$, there
		exists $\eta=\eta(\alpha,\beta,C)$ such that for every $\epsilon\in(0,\eta)$,
		\begin{align*}
			\scl_\alpha(\epsilon)
			\leq
			\scl_\beta(C\epsilon), \quad \scl_\alpha(\epsilon)
			\leq
			C\cdot\scl_\beta(\epsilon).
		\end{align*}
	\end{fact}

	The following family from \cite{Hel25a} provides a general class of admissible scaling families and
	contains the power scaling used below.

	\begin{example}\label{ex 2.3}
		For integers $p,q\geq1$, set
		\begin{align*}
			\scl_\alpha : \epsilon > 0 \mapsto \frac{1}{\exp^{\circ p}\left(\alpha \cdot \log_+^{\circ q}(\epsilon^{-1})\right)},
		\end{align*}
		where $\log_+(t)=\max\{\log t,0\}$ and $f^{\circ n}$ denotes the $n$-fold iterate. The choices $(p,q)=(1,1)$
		and $(2,1)$ give $\scl_\alpha(\epsilon)=\epsilon^\alpha$ and $\scl_\alpha(\epsilon)=e^{-\epsilon^{-\alpha}}$,
		respectively.
	\end{example}

Motivated by the idea of \cite{Bow73,FH12}, we can   define 	corresponding entropy quantities for a giving scaling by using   Carath\'{e}odory construction.

Let $K\subseteq X$
	be nonempty, not necessarily compact or $T$-invariant, and let $\scl=\{\scl_\alpha\}_{\alpha\geq0}$
	be a scaling. For $\epsilon>0$, $N\in\mathbb N$, and $\alpha\geq0$, define
	\begin{align*}
		M_{\scl_\alpha}(T,K,\epsilon,N)=\inf\sum_{i\in I}\limits \scl_\alpha(e^{-n_i}),
	\end{align*}
	where the infimum is taken over all finite or countable covers $\{B_{n_i}(x_i,\epsilon)\}_{i\in I}$
	of $K$ with $n_i\geq N$ and $x_i\in X$.
	Since $M_{\scl_\alpha}(T,K,\epsilon,N)$ is nondecreasing in $N$, the limit as $N\to\infty$ exists. We define
	\begin{align*}
		M_{\scl_\alpha}(T,K,\epsilon)=\lim_{N\to\infty}M_{\scl_\alpha}(T,K,\epsilon,N).
	\end{align*}
By the definition of the scaling,	the limit as $\epsilon\to 0$ also exists. We then define
	\begin{align*}
		M_{\scl_\alpha}(T,K):=\lim_{\epsilon\to 0}M_{\scl_\alpha}(T,K,\epsilon).
	\end{align*}

	\begin{proposition}
		\label{prop 2.4}
		For every $0<\beta<\alpha$, the following statements hold:
		\begin{align*}
			M_{\scl_\beta}(T,K)<+\infty
			&\quad\Longrightarrow\quad
			M_{\scl_\alpha}(T,K)=0,\\
			M_{\scl_\alpha}(T,K)>0
			&\quad\Longrightarrow\quad
			M_{\scl_\beta}(T,K)=+\infty.
		\end{align*}
		Consequently, there exists a critical value $\alpha_c\in[0,+\infty]$ such that
		\begin{align*}
			M_{\scl_\alpha}(T,K)
			=
			\begin{cases}
				+\infty, & 0<\alpha<\alpha_c,\\
				0,       & \alpha>\alpha_c.
			\end{cases}
		\end{align*}
	\end{proposition}

	\begin{proof}
		Fix $0<\beta<\alpha$. By Fact~\ref{fact 2.2}, for every $C>0$ and all sufficiently large $n$,
		\begin{align*}
			\scl_\alpha(e^{-n})
			\leq
			C\scl_\beta(e^{-n}).
		\end{align*}
		Taking the infimum over all covers, we obtain
		\begin{align}
			\label{equ 2.1}
			M_{\scl_\alpha}(T,K)
			\leq
			C M_{\scl_\beta}(T,K).
		\end{align}
		If $M_{\scl_\beta}(T,K)<+\infty$, the fact that $C>0$ is arbitrary in \eqref{equ 2.1} gives $M_{\scl_\alpha}(T,K)=0$.
		Similarly, if $M_{\scl_\alpha}(T,K)>0$, then
		\begin{align*}
			M_{\scl_\beta}(T,K)
			\geq
			\frac{1}{C}M_{\scl_\alpha}(T,K).
		\end{align*}
		Letting $C\to 0$ yields $M_{\scl_\beta}(T,K)=+\infty$.
		Choose $\alpha_c := \inf\left\{ \alpha>0: M_{\scl_\alpha}(T,K)=0 \right\},$ with the convention $\inf\emptyset=+\infty$.
	Then we finish the proof.
	\end{proof}

	\begin{definition}\label{def 2.5}
		The \emph{Bowen entropy scale of $T$ restricted to $K$} is defined by
		\begin{align*}
			h_{\scl}^{B}(T,K)
			=
			\sup\{\alpha>0:M_{\scl_\alpha}(T,K)=\infty\}
			=
			\inf\{\alpha>0:M_{\scl_\alpha}(T,K)=0\}.
		\end{align*}
		Here we use the conventions $\sup\emptyset=0$ and $\inf\emptyset=+\infty$.
	\end{definition}

	 If $\scl_\alpha(\epsilon)=\epsilon^\alpha$, then the above quantity coincides with the
	classical Bowen topological entropy (see \cite{Bow73,FH12}). We introduce some basic properties  of Bowen entropy scale as follows.

	\begin{proposition}\label{prop 2.6}
		\begin{enumerate}
			\item[\rm(1)] If $K_1\subseteq K_2$, then $h_{\scl}^{B}(T,K_1)\leq h_{\scl}^{B}(T,K_2)$.

			\item[\rm(2)] If $K=\bigcup_{i=1}^{\infty}K_i$, then
			\begin{align*}
				h_{\scl}^{B}(T,K)=\sup_{i\geq1}h_{\scl}^{B}(T,K_i).
			\end{align*}

			\item[\rm(3)] The value of $h_{\scl}^{B}(T,K)$ is independent of the choice of $d\in\mathcal D(X)$.

			\item[\rm(4)] Let $\pi:(X,T)\to(Y,S)$ be a factor map. Then
			\begin{align*}
				h_{\scl}^{B}\left(S,\pi(K)\right)
				\leq h_{\scl}^{B}(T,K).
			\end{align*}
			If $\pi$ is topologically conjugate, then equality holds.
		\end{enumerate}
	\end{proposition}

	\begin{proof}
	  (1) directly follows from the definition. For (2), we set $c:=\sup_i h_{\scl}^{B}(T,K_i)$. Then by monotonicity, we only need to show
	  \begin{align*}
	  	h_{\scl}^{B}(T,K)\leq \sup_{i\geq1}h_{\scl}^{B}(T,K_i).
	  \end{align*}
	   If $\beta>c$, then $M_{\scl_\beta}(T,K_i)=0$ for every $i$. Fix $\epsilon,\eta>0$ and $N\in\mathbb N$.
		For each $i$, choose a countable cover $\mathcal C_i$ of $K_i$ with  $n_{i,j}\geq N$ such
		that
		\begin{align*}
			\sum_{j\in J_i}\scl_\beta(e^{-n_{i,j}})<\eta 2^{-i}.
		\end{align*}
		Taking $\mathcal{C}:=\bigcup_{i=1}^{\infty}\mathcal{C}_i$, which forms  a countable cover of $K$. Then
		\begin{align*}
			M_{\scl_\beta}(T,K,\epsilon,N)
			\leq\sum_{i=1}^{\infty}\eta2^{-i}=\eta.
		\end{align*}
		Since $\eta>0$ is arbitrary, letting $N\to\infty$ and  $\epsilon\to 0$ yields $M_{\scl_\beta}(T,K)=0$.
		Letting $\beta\to c$ proves (2).

		For (3), we choose different metrics $d_1,d_2\in\mathcal D(X)$. Then
		for every $\delta>0$, there exists $\epsilon>0$ such that
		\begin{align*}
			(d_1)_n(x,y)<\epsilon \Longrightarrow (d_2)_n(x,y)<\delta
		\end{align*}
		for all $x,y\in X$. It implies
		$B_n(x,\epsilon,d_1)\subset B_n(x,\delta,d_2)$
		for all $x$ and $n$. Hence $M_{\scl_\alpha}(T,K,\delta,d_2)\leq M_{\scl_\alpha}(T,K,\epsilon,d_1)$,
		and therefore $h_{\scl}^{B}(T,K,d_2)\leq h_{\scl}^{B}(T,K,d_1)$. Interchanging the metrics gives
		equality.

		For (4), we denote by $B^X_n$ and $B^Y_n$ the corresponding Bowen ball on $X$ and $Y$, respectively.
		Fix $\epsilon>0$. By the uniform continuity of $\pi$, there exists $\delta>0$ such that
		\begin{align*}
			\pi\left(B_n^X(x,\delta)\right)
			\subset B_n^Y\left(\pi(x),\epsilon\right)
		\end{align*}
		for every $x\in X$ and $n\geq1$. Hence, one can show that
		\begin{align*}
			M_{\scl_\alpha}\left(S,\pi(K),\epsilon,N\right)
			\leq M_{\scl_\alpha}(T,K,\delta,N).
		\end{align*}
		Letting $N\to\infty$ and $\epsilon\to 0$, one can  show that
		 \begin{align*}
			h_{\scl}^{B}\left(S,\pi(K)\right)
			\leq h_{\scl}^{B}(T,K).
		\end{align*} If $\pi$ is a
		conjugacy, the same argument applied to $\pi^{-1}$ gives the reverse inequality.
	\end{proof}

	The above Bowen entropy scale can be viewed as a counterpart of Bowen entropy, which   motivates how to define a  scale quantity of capacity type.  If we define  capacity type entropy scale by fixing the parameter $\alpha$ and using spanning sets or separated sets as in \cite{Wal82}, this naturally leads us to consider the critical  value. However, under certain scalings, this may yield a rather trivial result. Therefore, we need to make appropriate adjustments to the definition.

	Let $K$ be a nonempty subset of $X$. For any $\epsilon>0$ and $n\in\mathbb N$, a set $E\subset X$
	is called an \emph{$(n,\epsilon)$-spanning set} of $K$ if $K\subset \bigcup_{x\in E}B_n(x,\epsilon).$
	We denote by   $r_n(K,\epsilon)$ the smallest cardinality of all $(n,\epsilon)$-spanning sets of $K$.
	A set $F\subset K$ is called an \emph{$(n,\epsilon)$-separated set} of $K$ if for any distinct $x,y\in F$, $d_n(x,y)\geq\epsilon$.
We denote by $s_n(K,\epsilon)$	the maximal cardinality of all $(n,\epsilon)$-separated sets of $K$.

	Let $\scl=\{\scl_\alpha\}_{\alpha\geq0}$ be a scaling. We define the following quantities:
	\begin{align*}
		\Lambda_{\scl_\alpha}(T,K,\epsilon,n)=&
		s_n(K,\epsilon)\cdot \scl_\alpha(e^{-n}),\\
		\overline{\Lambda}_{\scl_\alpha}(T,K)
		=&
		\lim_{\epsilon\to0}
		\limsup_{n\to\infty}
		\Lambda_{\scl_\alpha}(T,K,\epsilon,n).
	\end{align*}
	There is also a critical value at which $\overline{\Lambda}_{\scl_\alpha}(T,K)$
	changes from $+\infty$ to $0$.

	\begin{definition}\label{def 2.7}
		The \emph{upper capacity entropy scale of $T$ restricted to $K$} is defined by
		\begin{align*}
			\overline{h}_{\scl}^{C}(T,K)
			=
			\inf
			\left\{
			\alpha>0:
			\overline{\Lambda}_{\scl_\alpha}(T,K)=0
			\right\}
			=
			\sup
			\left\{
			\alpha>0:
			\overline{\Lambda}_{\scl_\alpha}(T,K)=\infty
			\right\}.
		\end{align*}
		When $K=X$, we omit the restriction of $K$.
	\end{definition}

	By replacing the $\limsup$ in the definition of $\overline{\Lambda}_{\scl_\alpha}$ with $\liminf$,
	one can define the lower capacity entropy scale   $\underline{h}_{\scl}^{C}(T,K)$.

	\begin{remark}\label{rem 2.8}
		For every nonempty subset $K\subset X$, since $s_n(K,2\epsilon) \leq r_n(K,\epsilon) \leq s_n(K,\epsilon),$ the values of $\underline{h}_{\scl}^{C}(T,K)$
		and $\overline{h}_{\scl}^{C}(T,K)$ do not change if they are defined via $r_n(K,\epsilon)$.
		Moreover, we can obtain the following inequalities:
		\begin{align*}
			h_{\scl}^{B}(T,K)
			\leq
			\underline{h}_{\scl}^{C}(T,K)
			\leq
			\overline{h}_{\scl}^{C}(T,K).
		\end{align*}
		However, there exist a scaling $\{e^{-(\log \frac{1}{t})^\alpha}\}$ such that the first inequality is strict \cite[Proposition~4.2]{KCML14}.
	\end{remark}

	Let
	$f:\mathbb N\to(0,+\infty)$ be a nondecreasing function with $f(n)\to+\infty$, and
	suppose that
	$\scl'_\alpha(e^{-n})=e^{-\alpha f(n)}$
	is induced by a scaling. Then
	$\overline h_{\scl}^{C}(T,X)$ coincides with the generalized topological
	entropy introduced in \cite{Gal03}:
	\begin{align*}
		h^f(T)
		=
		\lim_{\epsilon\to0}
		\limsup_{n\to\infty}
		\frac{\log s_n(X,\epsilon)}{f(n)}.
	\end{align*}
	The following example illustrates how the upper capacity
	entropy scale associated with this scaling distinguishes
	different levels of orbit complexity that are not detected by classical
	topological entropy.

\begin{example}\cite{RRS24}[Period-doubling cascade]
	\label{ex 2.9}
	Take $f(n)=\log(n+1)$.
	Consider the logistic maps
	$f_\lambda(x)=\lambda x(1-x)$ on $[0,1]$, where
	$0\leq\lambda\leq4$, and let $\lambda_\infty$ be
	the Feigenbaum parameter.
	\begin{itemize}
		\item If $1<\lambda<\lambda_\infty$ and $f_\lambda$ has an
		attracting cycle of period $2^r$, then
		\begin{align*}
			h_{top}(f_\lambda)=0,
			\quad
			\overline h_{\scl'}^{C}
			\left(f_\lambda,[0,1]\right)=r+1.
		\end{align*}

		\item At $\lambda=\lambda_\infty$,
		\begin{align*}
			h_{top}(f_{\lambda_\infty})=0,
			\quad
			\overline h_{\scl'}^{C}
			\left(f_{\lambda_\infty},[0,1]\right)=+\infty.
		\end{align*}
	\end{itemize}

	Thus this entropy scale distinguishes systems in the
	period-doubling cascade that all have zero topological entropy.
\end{example}

	We next establish some basic properties of upper capacity entropy scale.

	\begin{proposition}
		\label{prop 2.10}
		Let $K_1,K_2$ be nonempty subsets of $X$.
		\begin{enumerate}
			\item[\rm(1)] If $K_1\subseteq K_2$, then
			\begin{align*}
				\underline h_{\scl}^{C}(T,K_1)
				\leq \underline h_{\scl}^{C}(T,K_2),\quad
				\overline h_{\scl}^{C}(T,K_1)
				\leq \overline h_{\scl}^{C}(T,K_2).
			\end{align*}

			\item[\rm(2)] Let $\pi:(X,T)\to(Y,S)$ be a factor map. Then
			\begin{align*}
				\underline h_{\scl}^{C}\left(S,\pi(K_1)\right)
				\leq \underline h_{\scl}^{C}(T,K_1),\quad
				\overline h_{\scl}^{C}\left(S,\pi(K_1)\right)
				\leq \overline h_{\scl}^{C}(T,K_1).
			\end{align*}
			If $\pi$ is topologically conjugate, then both inequalities are equalities.

		\end{enumerate}
	\end{proposition}

	\begin{proof}
		  \rm(1) follows directly from the definitions. To prove \rm(2), we  use the
		  equivalent definition of capacity entropy scales by spanning sets.
		  Fix $\epsilon>0$. By the
		uniform continuity of $\pi$, there exists $\delta>0$ such that
		\begin{align*}
			d_X(x,x')<\delta
			\quad\Longrightarrow\quad
			d_Y\left(\pi(x),\pi(x')\right)<\epsilon
		\end{align*}
		for any $x,x'\in X$.
		Since $\pi\circ T=S\circ\pi$, the image under $\pi$ of an $(n,\delta)$-spanning set of $K_1$ is also  $(n,\epsilon)$-spanning
	 of $\pi(K_1)$. Hence, one has
		$r_n\left(\pi(K_1),\epsilon\right)
		\leq r_n(K_1,\delta),$
		which shows the desired results. If $\pi$ is  topological
		conjugate, then the reverse inequalities are obvious by applying $\pi^{-1}$.
	\end{proof}

	\subsection{Local entropy scales}
	\label{subsec 2.2}
	The local entropy introduced by Brin and Katok \cite{BK83} provides a measure-theoretic description
	of orbit complexity and is a basic tool in variational principles \cite{FH12,ZC16}. Motivated by
	this construction, we define local entropy scales for general scaling families.

	Let $\mu$ be a Borel probability measure on $(X,d)$ and let $\scl=\{\scl_\alpha\}_{\alpha\geq0}$ be
	a scaling. For $\epsilon>0$, define
	\begin{align*}
		\underline{h}_{\mu,\scl_\alpha}(T,x,\epsilon)
		&=
		\liminf_{n\to\infty}
		\frac{\scl_\alpha(e^{-n})}
		{\mu(B_n(x,\epsilon))},
		\\
		\overline{h}_{\mu,\scl_\alpha}(T,x,\epsilon)
		&=
		\limsup_{n\to\infty}
		\frac{\scl_\alpha(e^{-n})}
		{\mu(B_n(x,\epsilon))}.
	\end{align*}
	We set $\frac{\scl_\alpha(e^{-n})}{\mu(B_n(x,\epsilon))}=+\infty$ whenever $\mu(B_n(x,\epsilon))=0$.
	Similarly, there exist critical values at which these quantities jump from $+\infty$ to $0$.

	\begin{definition}\label{def 2.11}
		The \emph{upper and lower local entropy scales associated with $\mu$ at $x$} are defined by
		\begin{align*}
			\overline{h}_{\mu,\scl}(T,x,\epsilon)=&\inf \left\{\alpha>0:\overline{h}_{\mu,\scl_\alpha}(T,x,\epsilon)=0\right\}=\sup \left\{\alpha>0: \overline{h}_{\mu,\scl_\alpha}(T,x,\epsilon)=\infty\right\},\\
			\underline{h}_{\mu,\scl}(T,x,\epsilon)=&\inf \left\{\alpha>0:\underline{h}_{\mu,\scl_\alpha}(T,x,\epsilon)=0\right\}=\sup \left\{\alpha>0: \underline{h}_{\mu,\scl_\alpha}(T,x,\epsilon)=\infty\right\}.
		\end{align*}
		The \emph{upper and lower measure-theoretic entropy scales
			associated with $\mu$} are defined by
		\begin{align*}
			\overline{h}_{\mu,\scl}(T)=\int \overline{h}_{\mu,\scl}(T,x)\,d\mu(x)=\int\lim_{\epsilon\to 0}\overline{h}_{\mu,\scl}(T,x,\epsilon)\,d\mu(x), \\ \underline{h}_{\mu,\scl}(T)=\int \underline{h}_{\mu,\scl}(T,x)\,d\mu(x)=\int\lim_{\epsilon\to 0}\underline{h}_{\mu,\scl}(T,x,\epsilon)\,d\mu(x),
		\end{align*}
		respectively.
	\end{definition}

	\begin{remark}\label{rem 2.12}
		For fixed $n$ and $\epsilon$, the map $x\mapsto\mu(B_n(x,\epsilon))$ is lower semicontinuous and
		hence Borel measurable. Thus the liminf and limsup defining the local quantities are measurable.
	\end{remark}


	For the power scaling $\scl_\alpha(\epsilon)=\epsilon^\alpha$, Definition~\ref{def 2.11} reduces to
	the classical lower and upper Brin--Katok local entropies \cite{BK83}. More generally, suppose that
	there exist a function $b(n)\to+\infty$ and a continuous strictly increasing bijection $c:[0,+\infty)\to[0,+\infty)$
	such that
	\begin{align*}
		\frac{-\log\scl_\alpha(e^{-n})}{b(n)}
		\longrightarrow c(\alpha)
	\end{align*}
	for every $\alpha>0$. Then
	\begin{align*}
		\underline h_{\mu,\scl}(T,x)
		&=
		c^{-1}\left(
		\lim_{\epsilon\to0}
		\liminf_{n\to\infty}
		\frac{-\log\mu(B_n(x,\epsilon))}{b(n)}
		\right),\\
		\overline h_{\mu,\scl}(T,x)
		&=
		c^{-1}\left(
		\lim_{\epsilon\to0}
		\limsup_{n\to\infty}
		\frac{-\log\mu(B_n(x,\epsilon))}{b(n)}
		\right),
	\end{align*}
	where $c^{-1}(+\infty):=+\infty$. In particular, when $\scl_\alpha(e^{-n})=e^{-\alpha b(n)}$, these
	are the generalized local Bowen-ball decay rates considered in \cite{Gal07}.

	\subsection{A Billingsley-type theorem for entropy scales}
	We next relate local entropy scales to Bowen entropy scales. This comparison will provide the
	measure-theoretic lower bound in the variational principle.

	We begin with the following covering lemma.

	\begin{lemma} \cite{MW08}\label{lem 2.14}
		Let $\mathcal B=\{B_n(x,\epsilon):x\in X,\ n\in\mathbb N\}$. For every family $\mathcal F\subseteq\mathcal B$,
		there exists a possibly uncountable subfamily $\mathcal G\subseteq\mathcal F$ of pairwise disjoint
		Bowen balls such that
		\begin{align*}
			\bigcup_{B\in \mathcal F}B\subseteq \bigcup_{B_{n}(x,\epsilon)\in \mathcal G}B_{n}(x,3\epsilon).
		\end{align*}
	\end{lemma}

	\begin{theorem}[Billingsley-type theorem for Bowen entropy scales]
		\label{thm 2.15}
		Let $(X,T)$ be a TDS, let $E\subset X$ be a Borel set, and fix $\mu\in\mathscr M(X)$. For every $\alpha>0$,
		the following statements hold:
		\begin{enumerate}
			\item If $\underline{h}_{\mu,\scl}(T,x)\leq\alpha$ for every $x\in E$, then $h_{\scl}^{B}(T,E)\leq\alpha$;

			\item If $\mu(E)>0$ and $\underline{h}_{\mu,\scl}(T,x)\geq\alpha$ for every $x\in E$, then $h_{\scl}^{B}(T,E)\geq\alpha$.
		\end{enumerate}
	\end{theorem}

	\begin{proof}
		We prove the two assertions separately.

		(1) Fix $\beta>\alpha$ and set
		\begin{align*}
			E_m=\left\{x\in E:\underline{h}_{\mu,\scl}(T,x,\epsilon)<\beta, \forall \epsilon\in \left(0,\frac{1}{m}\right) \right\}.
		\end{align*}
		Then $E=\bigcup_{m=1}^{\infty}E_m$. Fix $m$ and $\epsilon\in(0,1/(3m))$. For each $x\in E_m$, the
		definition of the lower local entropy scale at level $\beta$ gives
		\begin{align*}
			\underline{h}_{\mu,\scl_{\beta}}(T,x,\epsilon)=\liminf_{n\to\infty}\frac{\scl_\beta(e^{-n})}{\mu(B_n(x,\epsilon))}=0.
		\end{align*}
		Thus, given $M,N\in\mathbb N$, there is a strictly increasing sequence $\{n_j(x)\}$ with $n_j(x)\geq N$
		such that $\frac{\scl_\beta(e^{-n_j(x)})}{\mu(B_{n_j(x)}(x,\epsilon))}<\frac{1}{M}$, equivalently,
		\begin{align}\label{equ 2.2}
			\mu(B_{n_j(x)}(x,\epsilon))>M\scl_\beta(e^{-n_j(x)}).
		\end{align}
		For $N\geq1$, set
		\begin{align*}
			\mathcal G_N
			:=\{B_{n_j(x)}(x,\epsilon):x\in E_m,\ n_j(x)\geq N\}.
		\end{align*}
		This is an open cover of $E_m$. By Lemma~\ref{lem 2.14}, there exists a pairwise disjoint subfamily $\mathcal G^* = \{B_{n_i}(x_i,\epsilon)\}_{i\in I} \subset \mathcal G_N$
		such that
		\begin{align*}
			E_m
			\subset
			\bigcup_{i\in I}B_{n_i}(x_i,3\epsilon).
		\end{align*}
		Since $X$ is second countable, this disjoint family is countable. By \eqref{equ 2.2},
		\begin{align*}
			M_{\scl_\beta}(T,E_m,3\epsilon,N)
			\leq
			\sum_{i\in I}\scl_\beta(e^{-n_i})
			\leq
			\frac1M\sum_{i\in I}\mu(B_{n_i}(x_i,\epsilon))
			\leq \frac1M.
		\end{align*}
		First letting $N\to\infty$ and then $M\to\infty$ gives $M_{\scl_\beta}(T,E_m,3\epsilon)=0$.
		Letting $\epsilon\to0$ yields $M_{\scl_\beta}(T,E_m)=0$, hence $h_{\scl}^{B}(T,E_m)\leq\beta$.
		Since $E=\bigcup_mE_m$, Proposition~\ref{prop 2.6} gives
		\begin{align*}
			h_{\scl}^{B}(T,E)=\sup_m h_{\scl}^{B}(T,E_m)\leq\beta.
		\end{align*}
		Letting $\beta\to\alpha$ proves the desired inequality.

		(2) Fix $0<\gamma<\alpha$ and define
		\begin{align*}
			E_m
			:=
			\left\{
			x\in E:
			\underline h_{\mu,\scl}
			\left(T,x,\frac1m\right)>\gamma
			\right\}.
		\end{align*}
		Since $\underline h_{\mu,\scl}(T,x) = \lim_{\epsilon\to0} \underline h_{\mu,\scl}(T,x,\epsilon) \geq\alpha>\gamma$
		for every $x\in E$, $E=\bigcup_{m=1}^{\infty}E_m.$ Since $\mu(E)>0$, there exists $M\in\mathbb N$
		such that $\mu(E_M)>0$.

		For every $x\in E_M$, we have
		\begin{align*}
			\underline h_{\mu,\scl_\gamma}
			\left(T,x,\frac1M\right)
			&=
			\liminf_{n\to\infty}
			\frac{\scl_\gamma(e^{-n})}
			{\mu(B_n(x,1/M))}
			=+\infty.
		\end{align*}
		For $N\in\mathbb N$, set
		\begin{align*}
			E_{N,M}
			:=
			\left\{
			x\in E_M:
			\frac{\scl_\gamma(e^{-n})}
			{\mu(B_n(x,1/M))}
			>1
			\text{ for every }n\geq N
			\right\}.
		\end{align*}
		Then $E_M=\bigcup_{N=1}^{\infty}E_{N,M}.$ Hence there exists $N_0\in\mathbb N$ such that $\mu(E_{N_0,M})>0$.

		Fix $N\geq N_0$ and $0<\epsilon<1/M$. Let $\left\{ B_{n_i}\left(y_i,\frac{\epsilon}{2}\right) \right\}_i$
		be an arbitrary cover of $E_{N_0,M}$ with $n_i\geq N$. After discarding any ball that does not
		meet $E_{N_0,M}$, choose $x_i \in E_{N_0,M} \cap B_{n_i}\left(y_i,\frac{\epsilon}{2}\right).$ By
		the triangle inequality,
		\begin{align*}
			B_{n_i}\left(y_i,\frac{\epsilon}{2}\right)
			\subset
			B_{n_i}(x_i,\epsilon)
			\subset
			B_{n_i}\left(x_i,\frac1M\right).
		\end{align*}
		Since $x_i\in E_{N_0,M}$ and $n_i\geq N\geq N_0$, $\mu\left(B_{n_i}\left(x_i,\frac1M\right)\right) < \scl_\gamma(e^{-n_i}).$
		Consequently,
		\begin{align*}
			0<\mu(E_{N_0,M})
			\leq
			\sum_i
			\mu\left(
			B_{n_i}\left(y_i,\frac{\epsilon}{2}\right)
			\right)
			\leq
			\sum_i\mu(B_{n_i}(x_i,\epsilon))
			\leq
			\sum_i\scl_\gamma(e^{-n_i}).
		\end{align*}
		Taking the infimum over all such covers yields
		\begin{align*}
			M_{\scl_\gamma}
			\left(
			T,E_{N_0,M},\frac{\epsilon}{2},N
			\right)
			\geq
			\mu(E_{N_0,M})>0.
		\end{align*}
		Letting $N\to\infty$ and then $\epsilon\to0$, we obtain
		$	M_{\scl_\gamma}(T,E_{N_0,M})
		\geq
		\mu(E_{N_0,M})>0.$
		Since $E_{N_0,M}\subset E$, monotonicity gives $M_{\scl_\gamma}(T,E)>0.$ Therefore, $h_{\scl}^{B}(T,E)\geq\gamma.$
		Since $0<\gamma<\alpha$ is arbitrary, it follows that
		$h_{\scl}^{B}(T,E)\geq\alpha.$
	\end{proof}

	\section{A variational principle for entropy scales on compact subsets}
	\label{sec 3}
	In this section, we introduce the weighted Bowen entropy scale and prove Theorem~\ref{thm 1.1}.

	Let $\scl=\{\scl_\alpha\}_{\alpha\geq0}$ be a scaling. For a bounded function $g:X\to\mathbb R$, $N\in\mathbb N$,
	and $\epsilon>0$, we define
	\begin{align*}
		W_{\scl_\alpha}(T,g,\epsilon,N)
		=
		\inf
		\sum_i c_i\scl_\alpha(e^{-n_i}),
	\end{align*}
	where the infimum is taken over all finite or countable families $\{(B_{n_i}(x_i,\epsilon),c_i)\}_{i\in I}$
	satisfying $c_i>0$, $n_i\geq N$, and $\sum_i c_i\chi_{B_{n_i}(x_i,\epsilon)}\geq g .$

	For a nonempty subset $K\subset X$, we write
	\begin{align*}
		W_{\scl_\alpha}(T,K,\epsilon,N)
		:=
		W_{\scl_\alpha}(T,\chi_K,\epsilon,N).
	\end{align*}
	Since $W_{\scl_\alpha}(T,K,\epsilon,N)$ is nondecreasing with respect to $N$ and nondecreasing as $\epsilon\to0$,
	we define
	\begin{align*}
		W_{\scl_\alpha}(T,K,\epsilon)
		=
		\lim_{N\to\infty}
		W_{\scl_\alpha}(T,K,\epsilon,N), \quad
		W_{\scl_\alpha}(T,K)
		=
		\lim_{\epsilon\to0}
		W_{\scl_\alpha}(T,K,\epsilon).
	\end{align*}
	As proved in Proposition~\ref{prop 2.4}, there exists a critical value at which $W_{\scl_\alpha}(T,K)$
	jumps from $+\infty$ to $0$.

	\begin{definition}\label{def 3.1}
		The \emph{weighted Bowen entropy scale} of $K$ is defined by
		\begin{align*}
			h_{\scl}^{WB}(T,K)
			=
			\sup\{\alpha>0:
			W_{\scl_\alpha}(T,K)=\infty\}
			=
			\inf\{\alpha>0:
			W_{\scl_\alpha}(T,K)=0\}.
		\end{align*}
	\end{definition}

The following lemma shows that the scaling axioms entail a summability property between different parameter levels. This property enables us to group Bowen balls according to their lengths and to control the cost of passing from weighted covers to ordinary covers.

\begin{lemma}\label{lem 3.2}
	Let $\scl$ be a scaling. For every $0<\alpha<\beta$,
	there exists a strictly increasing sequence
	$(n_k)_{k\geq1}$ of positive integers such that
	\begin{align*}
	\sum_{k=1}^{\infty}
	\frac{\scl_\beta(e^{-n_k})}
	{\scl_\alpha(e^{-n_{k+1}})}
	<\infty.
	\end{align*}
\end{lemma}

\begin{proof}
	For $\alpha>0$ and $n\in\mathbb N$, denote by
	\begin{align*}
	\ell_\alpha(n)=-\log\scl_\alpha(e^{-n}).
	\end{align*}
Then each$\ell_\alpha$ is 	 nondecreasing and tends to infinity as $n\to\infty$.

	Fix $\alpha<\gamma<\beta$. By definition,
	there exists $\eta>0$ such that
	$\scl_\beta(t)
	=o\left(\scl_\gamma(t)^{1+2\eta}\right)$
as $t\to 0$.	Consequently, for all sufficiently large $n$, one has
	\begin{align}\label{equ 3.1}
	\ell_\beta(n)\geq(1+2\eta)\ell_\gamma(n).
	\end{align}
Taking $C=e^{-1}$ in
Fact~\ref{fact 2.2}, we obatin
	\begin{align*}
	\scl_\gamma(e^{-n})
	\leq \scl_\alpha(e^{-(n+1)})
	\end{align*}
	for all sufficiently large $n$. Hence
	\begin{align}\label{equ 3.2}
	\ell_\alpha(n+1)\leq\ell_\gamma(n).
	\end{align}

Choose $n_1$ large enough such that  \eqref{equ 3.1} and \eqref{equ 3.2} hold for each $n\geq n_1$ and 	$\ell_\gamma(n_1)>0$. Define recursively
	\begin{align*}
	n_{k+1}
	=
	\min\left\{
	n>n_k:
	\ell_\gamma(n)\geq
	(1+\eta)\ell_\gamma(n_k)
	\right\}.
	\end{align*}
	This minimum exists since $\ell_\gamma(n)\to\infty$.
Then by \eqref{equ 3.2}, we can show
	\begin{align*}
	\ell_\alpha(n_{k+1})
	\leq \ell_\gamma(n_{k+1}-1)
	<
	(1+\eta)\ell_\gamma(n_k)
	\end{align*}
	and $\ell_\gamma(n_k)
	\geq
	(1+\eta)^{k-1}\ell_\gamma(n_1).$
	Therefore,
	\begin{align*}
		\frac{\scl_\beta(e^{-n_k})}
		{\scl_\alpha(e^{-n_{k+1}})}
		=e^{\ell_\alpha(n_{k+1})-\ell_\beta(n_k)
		}\leq e^{-\eta\ell_\gamma(n_k)}.
	\end{align*}
Thus
\begin{align*}
\sum_{k=1}^{\infty}
\frac{\scl_\beta(e^{-n_k})}
{\scl_\alpha(e^{-n_{k+1}})}
\leq
\sum_{k=1}^{\infty}
e^{-\eta\ell_\gamma(n_k)}
<\infty,
\end{align*}
which finishes the proof.
\end{proof}

	By considering $c_i=1$, it immediately shows $h_{\scl}^{WB}(T,Z)\leq h_{\scl}^{B}(T,Z)$. For the reverse inequality, we group Bowen balls using
	the sequence supplied by Lemma~\ref{lem 3.2} and use maximal separated sets to compare
	the weighted and ordinary covering costs. In this way, our approach differs from that in \cite{FH12}  which relying on the a Vitali-type covering lemma.

\begin{proposition}\label{prop 3.3}
	Let $(X,T)$ be a TDS and let $\scl$ be a scaling.
	For every nonempty subset $Z\subset X$,
	\begin{align*}
	h_{\scl}^{WB}(T,Z)=h_{\scl}^{B}(T,Z).
	\end{align*}
\end{proposition}

\begin{proof}
We only need to prove $h_{\scl}^{WB}(T,Z)\geq h_{\scl}^{B}(T,Z)$.
Fix $0<\alpha<\beta$.
	By Lemma~\ref{lem 3.2}, choose a strictly increasing
	sequence $(n_k)$ such that
	\begin{align*}
	\vartheta_k
	=
	\frac{\scl_\beta(e^{-n_k})}
	{\scl_\alpha(e^{-n_{k+1}})},
	\qquad
	Q=\sum_{k=1}^{\infty}\vartheta_k<\infty.
	\end{align*}
	Fix $\epsilon>0$ and $K\in\mathbb N$.
	Consider a finite or countable weighted cover
	$\left\{
	\bigl(B_{m_i}(x_i,\epsilon),c_i\bigr)
	\right\}_{i\in I}$
	of $Z$, where $m_i\geq n_K$ and $c_i>0$.
	For each $k\geq K$, put
	\begin{align*}
	I_k=\{i\in I:n_k\leq m_i<n_{k+1}\}
	\end{align*}
and $\widehat B_i=B_{n_k}(x_i,\epsilon)$	 for $i\in I_k$. Since $B_{m_i}(x_i,\epsilon)\subset\widehat B_i$,
one can show
	\begin{align*}
	\sum_{k\geq K}\sum_{i\in I_k}
	c_i\chi_{\widehat B_i}
	\geq \chi_Z.
	\end{align*}

	For $0<t<1$, we define
	\begin{align*}
	Z_k=
	\left\{
	x\in Z:
	\sum_{i\in I_k}c_i\chi_{\widehat B_i}(x)>\frac{t\vartheta_k}{Q}
	\right\}.
	\end{align*}
	Then $Z=\bigcup_{k\geq K}Z_k$ as $\sum_{k\geq K}\frac{t\vartheta_k}{Q}\leq t<1$.
	For each nonempty $Z_k$, choose a finite maximal
	$(n_k,2\epsilon)$-separated subset $F_k\subset Z_k$. It implies
	\begin{align*}
	Z_k\subset
	\bigcup_{y\in F_k}B_{n_k}(y,2\epsilon).
	\end{align*}
	Each $\widehat B_i$ contains at most one point of $F_k$.
	It follows that
	\begin{align*}
		|F_k|\frac{t\vartheta_k}{Q}
		<
		\sum_{y\in F_k}\sum_{i\in I_k}
		c_i\chi_{\widehat B_i}(y)
		\leq
		\sum_{i\in I_k}c_i.
	\end{align*}
	Consequently,
	\begin{align*}
		M_{\scl_\beta}(T,Z_k,2\epsilon,n_K)
		\leq
		|F_k|\scl_\beta(e^{-n_k})
		\leq
		\frac{Q}{t}\,
		\scl_\alpha(e^{-n_{k+1}})
		\sum_{i\in I_k}c_i
		\leq
		\frac{Q}{t}
		\sum_{i\in I_k}c_i\scl_\alpha(e^{-m_i}).
	\end{align*}
	Taking the union of these covers, by Proposition \ref{prop 2.6}, one can show
	\begin{align*}
	M_{\scl_\beta}(T,Z,2\epsilon,n_K)
	\leq
	\frac{Q}{t}
	\sum_{i\in I}c_i\scl_\alpha(e^{-m_i}).
	\end{align*}
	Taking the infimum over all weighted covers, we obtain
	\begin{align*}
	M_{\scl_\beta}(T,Z,2\epsilon,n_K)
	\leq
	\frac{Q}{t}\,
	W_{\scl_\alpha}(T,Z,\epsilon,n_K).
	\end{align*}
	Letting $t\to 1$, $K\to\infty$, and
	$\epsilon\to 0$, we conclude that
	$M_{\scl_\beta}(T,Z)
	\leq
	Q\,W_{\scl_\alpha}(T,Z).$

If $h_{\scl}^{WB}(T,Z)=\infty$, then $h_{\scl}^{WB}(T,Z)\geq h_{\scl}^{B}(T,Z)$ is obvious.	Suppose that $h_{\scl}^{WB}(T,Z)<\infty$.
	For any $\beta>h_{\scl}^{WB}(T,Z)$, we choose $h_{\scl}^{WB}(T,Z)<\alpha<\beta.$
	Then one can show $M_{\scl_\beta}(T,Z)=0$.
	Therefore, $h_{\scl}^{B}(T,Z)\leq\beta.$
	Letting $\beta\to  h_{\scl}^{WB}(T,Z)$, we can  finally prove
\begin{align*}
h_{\scl}^{WB}(T,Z)=h_{\scl}^{B}(T,Z).
\end{align*}

\end{proof}

	The following weighted Frostman lemma plays an essential role in the proof of
	Theorem~\ref{thm 1.1}.

	\begin{lemma}\label{lem 3.5}
		Let $K\subset X$ be a nonempty compact set, let $\alpha\geq0$, $N\in\mathbb N$, and $\epsilon>0$.
		Suppose that
		$	W_{\scl_\alpha}(T,K,\epsilon,N)>0.$
		Then there exists $\mu\in\mathscr M(X)$ such that $\mu(K)=1$ and
		\begin{equation*}
			\mu(B_n(x,\epsilon))\le \frac{\scl_\alpha(e^{-n})}{W_{\scl_\alpha}(T,K,\epsilon,N)}
		\end{equation*}
		for every $x\in X$ and every $n\ge N$.
	\end{lemma}

	\begin{proof}
		Set $c=W_{\scl_\alpha}(T,K,\epsilon,N).$ Since $K$ is compact, $c<\infty$. Define a functional $p:C(X,\mathbb R)\to\mathbb R$
		by
		\begin{align*}
			p(f)=\frac{W_{\scl_\alpha}(T,\chi_K f,\epsilon,N)}{c}.
		\end{align*}
		It follows directly from the definition of $W_{\scl_\alpha}$ that
		\begin{enumerate}
			\item $p(f+g)\le p(f)+p(g)$ for all $f,g\in C(X,\mathbb R)$; \item $p(tg)=t\,p(g)$ for all $t\ge 0$
			and $g\in C(X,\mathbb R)$; \item $p(\chi)=1$, where $\chi$ denotes the constant function equal to $1$;
			\item $0\le p(g)\le \|g\|_\infty$ for every $g\in C(X,\mathbb R)$, and $p(f)=0$ whenever $f\le 0$.
		\end{enumerate}

		Define a linear functional on the constant functions by $L(b\chi)=b$. Since $L\leq p$ on this
		subspace, the Hahn--Banach theorem gives a linear extension $\Psi:C(X,\mathbb R)\to\mathbb R$
		satisfying
		\begin{align*}
			-p(-f)\leq \Psi(f)\leq p(f)
		\end{align*}
		for every $f\in C(X,\mathbb R)$. In particular, $\Psi$ is positive. Hence, by the Riesz
		representation theorem, there exists a Borel probability measure $\mu\in\mathscr M(X)$ such that
		\begin{align*}
			\Psi(f)=\int_X f\,d\mu
		\end{align*}
		for every $f\in C(X,\mathbb R)$.

		We first prove that $\mu(K)=1$. Let $E\subset X\setminus K$ be compact. By Urysohn's lemma, there
		exists $f\in C(X,\mathbb R)$ such that $0\leq f\leq1,\quad f|_K=0,\quad f|_E=1$. Then $p(f)=0$,
		and therefore
		\begin{align*}
			\mu(E)\leq \int_X f\,d\mu =\Psi(f)\leq
			p(f)=0.
		\end{align*}
		By the inner regularity of $\mu$, we obtain $\mu(X\setminus K)=0$, and hence $\mu(K)=1.$

		It remains to prove the estimate on Bowen balls. Fix $x\in X$ and $n\geq N$. Let $E\subset B_n(x,\epsilon)$
		be compact. By Urysohn's lemma, there exists $g\in C(X,\mathbb R)$ such that
		\begin{align*}
			0\leq g\leq1,\quad
			g|_E=1,\quad
			g|_{X\setminus B_n(x,\epsilon)}=0 .
		\end{align*}
		Then $\mu(E) \leq \Psi(g) \leq p(g).$ Since $\chi_K g\leq \chi_{B_n(x,\epsilon)}$, the definition
		of $W_{\scl_\alpha}$ gives
		\begin{align*}
			W_{\scl_\alpha}(T,\chi_K g,\epsilon,N)
			\leq
			\scl_\alpha(e^{-n}).
		\end{align*}
		Consequently,
		\begin{align*}
			\mu(E)
			\leq
			p(g)
			\leq
			\frac{\scl_\alpha(e^{-n})}
			{W_{\scl_\alpha}(T,K,\epsilon,N)} .
		\end{align*}
		Taking the supremum over all compact subsets $E\subset B_n(x,\epsilon)$ and using the inner
		regularity of $\mu$ yields
		\begin{align*}
			\mu(B_n(x,\epsilon))
			\leq
			\frac{\scl_\alpha(e^{-n})}
			{W_{\scl_\alpha}(T,K,\epsilon,N)} .
		\end{align*}
		This completes the proof.
	\end{proof}

	We now prove Theorem~\ref{thm 1.1}.

	\begin{proof}[Proof of Theorem~\ref{thm 1.1}]
		We first prove the lower bound. Let $\mu\in\mathscr M(X)$ satisfy $\mu(K)=1$. If $\underline{h}_{\mu,\scl}(T)=0$,
		the required inequality is immediate. Assume that $\underline{h}_{\mu,\scl}(T)>0.$ Fix $0<\delta<\underline{h}_{\mu,\scl}(T)$.
		Define
		\begin{align*}
			K_0
			=
			\left\{
			x\in K:
			\underline{h}_{\mu,\scl}(T,x)
			\geq
			\underline{h}_{\mu,\scl}(T)-\delta
			\right\}.
		\end{align*}
		Then $\mu(K_0)>0$. Otherwise, $\underline{h}_{\mu,\scl}(T,x) < \underline{h}_{\mu,\scl}(T)-\delta$
		for $\mu$-almost every $x\in K$, which contradicts the definition of $\underline{h}_{\mu,\scl}(T)$
		as the integral of the pointwise local entropy scale.

		Applying Theorem~\ref{thm 2.15}, we obtain
		\begin{align*}
			h_{\scl}^{B}(T,K)
			\geq
			h_{\scl}^{B}(T,K_0)\geq
			\underline{h}_{\mu,\scl}(T)-\delta .
		\end{align*}
		Letting $\delta\to0$ gives $	h_{\scl}^{B}(T,K)
		\geq
		\underline{h}_{\mu,\scl}(T)$
		and thus
		\begin{align*}
			h_{\scl}^{B}(T,K)
			\geq
			\sup
			\left\{
			\underline{h}_{\mu,\scl}(T):
			\mu\in\mathscr M(X),\mu(K)=1
			\right\}.
		\end{align*}
		The case $\underline{h}_{\mu,\scl}(T)=+\infty$ follows by replacing $\underline{h}_{\mu,\scl}(T)-\delta$
		with an arbitrary finite positive number.

		We next prove the reverse inequality. Assume that $h_{\scl}^{B}(T,K)>0$. By
		Proposition~\ref{prop 3.3}, it is enough to work with the weighted entropy scale. Take $	0<\alpha<h_{\scl}^{WB}(T,K).$
		Then there exist $\epsilon>0$ and $N\in\mathbb N$ such that $W_{\scl_\alpha}(T,K,\epsilon,N)>0$.
		By Lemma~\ref{lem 3.5}, there exists a probability measure $\mu\in\mathscr M(X)$ satisfying $\mu(K)=1$
		and
		\begin{align*}
			\mu(B_n(x,\epsilon))
			\leq
			\frac{\scl_\alpha(e^{-n})}
			{W_{\scl_\alpha}(T,K,\epsilon,N)}
		\end{align*}
		for every $x\in X$ and $n\geq N$. Therefore,
		\begin{align*}
			\underline{h}_{\mu,\scl_\alpha}(T,x,\epsilon)
			=
			\liminf_{n\to\infty}
			\frac{\scl_\alpha(e^{-n})}
			{\mu(B_n(x,\epsilon))}\geq
			W_{\scl_\alpha}(T,K,\epsilon,N)>0 .
		\end{align*}
		By the definition of the local entropy scale, this implies
	$\underline{h}_{\mu,\scl}(T,x)\geq\alpha .$
		Hence $\underline h_{\mu,\scl}(T)\geq\alpha$. Letting $\alpha\to h_{\scl}^{WB}(T,K)$ and
		using Proposition~\ref{prop 3.3}, we obtain
		\begin{align*}
			h_{\scl}^{B}(T,K)=h_{\scl}^{WB}(T,K)
			\leq
			\sup
			\left\{
			\underline{h}_{\mu,\scl}(T):
			\mu\in\mathscr M(X),\mu(K)=1
			\right\}.
		\end{align*}
		This completes the proof.
	\end{proof}

	\section{Factor inequalities for entropy scales}\label{sec 4}
	In this section, we prove the factor inequalities in Theorem~\ref{thm 1.2} and compute the
	associated transfer function for several standard scaling families.

	\subsection{Upper capacity and Bowen factor inequalities}
	\label{subsec 4.1}

	Let $(X,d_X,T)$ and $(Y,d_Y,S)$ be topological dynamical systems. A continuous surjection $\pi:X\longrightarrow Y$
	is called a \emph{factor map} if $\pi\circ T=S\circ\pi$. In this case, $(Y,S)$ is a factor of $(X,T)$.

	For a nonempty subset $A\subset X$, we denote the classical upper capacity entropy of $A$ by
	\begin{align*}
		h_{top}(T,A)
		:=
		\lim_{\epsilon\to0}
		\limsup_{n\to\infty}
		\frac{1}{n}\log r_n(A,\epsilon).
	\end{align*}
	 Given a factor map $\pi$, set
	\begin{align*}
		b_\pi
		:=
		\sup_{y\in Y}
		h_{top}
		\left(T,\pi^{-1}(y)\right).
	\end{align*}

	The following   estimate provides a more precise relation between
	the covering numbers of a set and  its image, which is obtained by adapting Bowen's classical
	argument to Bowen metric \cite{Bow71}. To distinguish the Bowen balls under different dynamical systems, we denote by $B_n^Y(y,\delta)$ the Bowen ball in the metric space $(Y,d_Y)$.

	\begin{lemma}
		\label{lem 4.1}
		For every $\epsilon>0$ and $\tau>0$, there exist $\delta>0$ and $M\in\mathbb N$
		such that, for every $y\in Y$ and $n\in\mathbb N$, the set $\pi^{-1}\left(B_n^Y(y,\delta)\right)$
		can be covered by at most $e^{(b_\pi+\tau)(n+M)}$ $n$-th Bowen balls in $X$ with radius $4\epsilon$.
		Consequently, for every nonempty subset $E\subset X$,
		\begin{align*}
			r_n(E,4\epsilon)
			\leq
			e^{(b_\pi+\tau)(n+M)}r_n\left(\pi(E),\delta\right).
		\end{align*}
	\end{lemma}

	\begin{proof}
		Assume that $b_\pi<+\infty$. Otherwise, there is nothing to prove.	Fix $\epsilon>0$ and $\tau>0$. For each $y\in Y$, we can choose $m(y)\in\mathbb N$ such that
		\begin{align}
			\label{equ 4.1}
			r_{m(y)}\left(\pi^{-1}(y),\epsilon\right)
			\leq e^{(b_\pi+\tau)m(y)}.
		\end{align}
		Let $F_y$ be an $(m(y),\epsilon)$-spanning set of $\pi^{-1}(y)$ with  the smallest cardinality and
		put
		\begin{align*}
			U_y
			:=
			\bigcup_{z\in F_y}B_{m(y)}^X(z,2\epsilon).
		\end{align*}
		Then $U_y$ is an open neighborhood of $\pi^{-1}(y)$ and thus $X\setminus U_y$ is compact. Since  $y\notin\pi(X\setminus U_y)$,
		there exists an open neighborhood $W_y$ of $y$ such that $\pi^{-1}(W_y)\subset U_y.$ Hence, $$Y=\bigcup_{y\in Y}W_y.$$

		By the compactness of $Y$, one can choose finite	C $y_1,\ldots,y_q\in Y$ such that
		$Y=\bigcup_{i=1}^qW_{y_i}.$
		By the Lebesgue number lemma, there exists $\lambda>0$ such that every
		subset of $Y$ of diameter less than $\lambda$ is contained in some
		$W_{y_i}$. Fix $0<\delta<\lambda/2$. Then every ball
		$B_Y(y,\delta)$ is contained in one of the sets $W_{y_i}$.
		Set
		\begin{align*}
			M:=\max_{1\leq i\leq q}m(y_i).
		\end{align*}
		Fix $y\in Y$ and $n\in\mathbb N$. We now proceed to define the time series $\{t_i\}_{i\geq 0}$. Set $t_0=0$. If $t_s<n$ for some $s$, we choose $i_s\in\{1,\ldots,q\}$
		such that $B_Y\left(S^{t_s}y,\delta\right) \subset W_{y_{i_s}},$ and define $t_{s+1}:=t_s+m(y_{i_s}).$
		Continue until $t_\ell<n\leq t_{\ell+1}$. Then
		\begin{align}
			\label{equ 4.2}
			t_{\ell+1}\leq n+M.
		\end{align}

		Take $x\in\pi^{-1}(B_n^Y(y,\delta))$. For every $0\leq s\leq\ell$, one has
		\begin{align*}
			\pi(T^{t_s}x)
			=
			S^{t_s}\pi(x)
			\in
			B_Y\left(S^{t_s}y,\delta\right)
			\subset W_{y_{i_s}}.
		\end{align*}
		Then there exists $z_s\in F_{y_{i_s}}$ such that $	d_{m(y_{i_s})}^X(T^{t_s}x,z_s)<2\epsilon.$
		Thus every point of $\pi^{-1}(B_n^Y(y,\delta))$ determines a tuple $(z_0,\ldots,z_\ell)\in \prod_{s=0}^{\ell}F_{y_{i_s}}$. If two
		points determine the same tuple, then their $d_n^X$-distance is less than $4\epsilon$ by the triangle inequality.
		Hence each nonempty class is contained in an
		$(n,4\epsilon)$-Bowen ball, and
		\begin{align*}
			r_n\left(
			\pi^{-1}\left(B_n^Y(y,\delta)\right),4\epsilon
			\right)
			\leq
			\prod_{s=0}^{\ell}\left|F_{y_{i_s}}\right|
			\leq e^{(b_\pi+\tau)
				\sum_{s=0}^{\ell}m(y_{i_s})}
			=e^{(b_\pi+\tau)t_{\ell+1}}\leq e^{(b_\pi+\tau)(n+M)}.
		\end{align*}
		This proves the first assertion.

		Let $G$ be a  $(n,\delta)$-spanning set of $\pi(E)$ with  the smallest cardinality. Then
		\begin{align*}
			E\subset \bigcup_{y\in G}\pi^{-1}(B_n^Y(y,\delta)).
		\end{align*}
		Hence,
		\begin{align*}
			r_n(E,4\epsilon)\leq \sum_{y\in G}r_n\left(
			\pi^{-1}\left(B_n^Y(y,\delta)\right),4\epsilon
			\right)\leq r_n(\pi(E),\delta)e^{(b_\pi+\tau)(n+M)}.
		\end{align*}
	\end{proof}

	Classical topological entropy measures the exponential growth rate of
	covering numbers. Thus the exponential fiber term $e^{b_\pi n}$ gives
	the additive correction $b_\pi$ in the classical factor inequality. By
	contrast, entropy scales are defined through a dimension-type critical
	value, and multiplication by $e^{b_\pi n}$ need not change the critical
	index by $b_\pi$. So we introduce a new notion called   \emph{transfer function} $\Gamma$ to describe the fiber contribution under a factor map.

For any $u,b\in[0,+\infty)$, define
	\begin{align}\label{equ 4.30}
		\Gamma_{\scl}(u,b) := \inf\left\{\gamma>0: \exists \alpha>u, \tau>0 \text{ such that} \limsup_{n\to\infty} e^{(b+\tau)n} \frac{\scl_\gamma(e^{-n})} {\scl_\alpha(e^{-n})} <+\infty \right\}.
	\end{align}
	For the scaling $\scl_\alpha(e^{-n})=e^{-n\alpha}$, since
	\begin{align*}
		e^{(b+\tau)n}
		\frac{\scl_\gamma(e^{-n})}
		{\scl_\alpha(e^{-n})}
		=
		e^{(b+\tau+\alpha-\gamma)n},
	\end{align*}
	one can show that under such case, $\Gamma_{\scl}(u,b)=u+b.$
	 Hence, this is a well-defined notion.
	Now we proceed to prove Theorem~\ref{thm 1.2}.
	\begin{proof}[Proof of Theorem~\ref{thm 1.2}]
	Set $b_\pi
	:=
	\sup_{y\in Y}
	h_{top}
	\left(T,\pi^{-1}(y)\right)=\sup_{y\in Y}
	h_{top}^B
	\left(T,\pi^{-1}(y)\right)$.
	We divide the proof into two steps.
		\begin{step1}
			We show that
			\begin{align*}
				\overline h_{\scl}^{C}\left(S,\pi(E)\right)
				\leq
				\overline h_{\scl}^{C}(T,E)\leq \Gamma_{\scl}\left(
				\overline h_{\scl}^{C}\left(S,\pi(E)\right),b_\pi
				\right).
			\end{align*}
		\end{step1}
		The left-hand inequality follows from Proposition~\ref{prop 2.10}. For the right-hand inequality, we set
		\begin{align*}
			u:=\overline h_{\scl}^{C}\left(S,\pi(E)\right).
		\end{align*}
		Suppose that both $u$ and $b_\pi$ are finite. Let $\gamma$ be admissible in  $\Gamma_{\scl}(u,b_\pi)$. That is there exist $\alpha>u$, $\tau>0$, $C>0$, and $N_0\in\mathbb N$ such that for every $n\geq N_0$.
		\begin{align}
			\label{equ 4.3}
			e^{(b_\pi+\tau)n}
			\frac{\scl_\gamma(e^{-n})}
			{\scl_\alpha(e^{-n})}
			\leq C.
		\end{align}

		Fix $\epsilon>0$. By Lemma~\ref{lem 4.1}, there exist $\delta>0$ and $M\in\mathbb N$ such that,
		for every $n\in\mathbb N$,
		\begin{align*}
			r_n(E,4\epsilon)
			\leq
			e^{(b_\pi+\tau)(n+M)}\cdot r_n\left(\pi(E),\delta\right).
		\end{align*}
		Since $\alpha>u$, by Remark~\ref{rem 2.8}
		we have $\limsup_{n\to\infty}
		r_n\left(\pi(E),\delta\right)
		\scl_\alpha(e^{-n})
		=0$.
		For $n\geq N_0$, by \eqref{equ 4.3}, we obtain
		\begin{align*}
			r_n(E,4\epsilon)\scl_\gamma(e^{-n})
			&\leq
			e^{(b_\pi+\tau)M}\cdot
			r_n\left(\pi(E),\delta\right)
			e^{(b_\pi+\tau)n}\scl_\gamma(e^{-n})\\
			&\leq
			Ce^{(b_\pi+\tau)M}\cdot
			r_n\left(\pi(E),\delta\right)
			\scl_\alpha(e^{-n}).
		\end{align*}
		Therefore,
		\begin{align*}
			\limsup_{n\to\infty}
			r_n(E,4\epsilon)\scl_\gamma(e^{-n})=0.
		\end{align*}
		Letting $\epsilon\to0$ yields that $\overline h_{\scl}^{C}(T,E)\leq\gamma$. Taking the infimum over all
		admissible $\gamma$ proves
		\begin{align*}
			\overline h_{\scl}^{C}(T,E)
			\leq
			\Gamma_{\scl}(u,b_\pi).
		\end{align*}

		\begin{step2}
			We prove
			\begin{align*}
				h_{\scl}^{B}\left(S,\pi(E)\right)
				\leq h_{\scl}^{B}(T,E)\leq \Gamma_{\scl}\left(
				h_{\scl}^{B}\left(S,\pi(E)\right),b_\pi
				\right).
			\end{align*}

		\end{step2}

		By Proposition~\ref{prop 2.6}, we only need to prove the right-hand inequality. Denote by
		\begin{align*}
			u:=h_{\scl}^{B}\left(S,\pi(E)\right).
		\end{align*}
		We suppose that both $u,b_\pi$ are finite.
		Let $\gamma$ be admissible in \eqref{equ 4.30}, and choose $\alpha>u$, $\tau>0$, $C>0$,
		and $N_0\in\mathbb N$ so that \eqref{equ 4.3} holds.

		Fix $\epsilon>0$, and let $\delta>0$ and $M\in\mathbb N$ be given by Lemma~\ref{lem 4.1}. Let $N\geq N_0$,
		and consider
		an countable cover of $\pi(E)$:
		\begin{align*}
			\pi(E)\subset \bigcup_{j\in J}B_{n_j}^Y(y_j,\delta),
		\end{align*} where $n_j\geq N$. For each $j$,  Lemma~\ref{lem 4.1} shows that each  $\pi^{-1}(B_{n_j}^Y(y_j,\delta))$
		can by covered by at most $e^{(b_\pi+\tau)(n_j+M)}$ $n_j$-th Bowen balls in $X$ of  radius $4\epsilon$.
		Therefore,
		\begin{align*}
			M_{\scl_\gamma}(T,E,4\epsilon,N)
			\leq
			\sum_{j\in J}
			e^{(b_\pi+\tau)(n_j+M)}
			\scl_\gamma(e^{-n_j})
			\leq
			Ce^{(b_\pi+\tau)M}
			\sum_{j\in J}
			\scl_\alpha(e^{-n_j}).
		\end{align*}
		Taking the infimum over all such covers and then letting $N\to\infty$, we obtain
		\begin{align*}
			M_{\scl_\gamma}(T,E,4\epsilon)
			\leq
			Ce^{(b_\pi+\tau)M}
			M_{\scl_\alpha}\left(S,\pi(E),\delta\right).
		\end{align*}
		Since
		$\alpha>h_{\scl}^{B}(S,\pi(E))$, one has $M_{\scl_\alpha}\left(S,\pi(E),\delta\right)=0$. Therefore,
		\begin{align*}
			M_{\scl_\gamma}(T,E,4\epsilon)=0.
		\end{align*}
		Letting $\epsilon\to0$, we obtain
		$M_{\scl_\gamma}(T,E)=0$, and hence
		$h_{\scl}^{B}(T,E)\leq\gamma$.
		Taking the infimum over all admissible $\gamma$ completes the proof.
	\end{proof}

	\begin{corollary}
		\label{cor 4.3}
		Suppose that $\scl_\alpha(t)=t^\alpha$. Then, for every
		nonempty subset $E\subset X$,
		\begin{align*}
			\overline h_{\scl}^{C}\left(S,\pi(E)\right)
			\leq
			\overline h_{\scl}^{C}(T,E)
			\leq
			\overline h_{\scl}^{C}\left(S,\pi(E)\right)
			+
			\sup_{y\in Y}
			\overline h_{\scl}^{C}
			\left(T,\pi^{-1}(y)\right),
		\end{align*}
		and
		\begin{align*}
			h_{\scl}^{B}\left(S,\pi(E)\right)
			\leq
			h_{\scl}^{B}(T,E)
			\leq
			h_{\scl}^{B}\left(S,\pi(E)\right)
			+
			\sup_{y\in Y}
			h_{\scl}^{B}
			\left(T,\pi^{-1}(y)\right).
		\end{align*}
	\end{corollary}

	The first inequality extends Bowen's classical factor inequality
	\cite{Bow71} to arbitrary
	nonempty subsets. The second is the classical factor inequality of Bowen entropy
	\cite{FHYZ12,OZ11}.

	\subsection{Transfer functions for standard scalings}

	For certain particular scalings, the transfer function
	$\Gamma_{\scl}$ admits explicit computations, from which its concrete form can be determined.

	\begin{proposition}\label{prop 4.4}
		Let $u,b\in[0,+\infty)$. For each of the following scaling families, the transfer function defined
		in \eqref{equ 4.30} is given as follows:
		\begin{enumerate}
			\item[\rm(1)] Let $\phi: [0,+\infty)\to [0,+\infty)$ be a strictly increasing homeomorphism. For
			scaling $\scl_\alpha(t)=t^{\phi(\alpha)}$, one has
			\begin{align*}
				\Gamma_{\scl}(u,b)
				=
				\phi^{-1}\left(\phi(u)+b\right).
			\end{align*}

			\item[\rm(2)] For $\theta>0$ and $\scl_\alpha(t) =e^{-\alpha\left(\log\frac1t\right)^\theta},$
			one has
			\begin{align*}
				\Gamma_{\scl}(u,b)
				=
				\begin{cases}
					+\infty, & 0<\theta<1,\\
					u+b,     & \theta=1,\\
					u,       & \theta>1.
				\end{cases}
			\end{align*}

			\item[\rm(3)] For the logarithmic scaling $\scl_\alpha(t) = \left(\log\frac1t\right)^{-\alpha},$
			one has
			\begin{align*}
				\Gamma_{\scl}(u,b)=+\infty.
			\end{align*}

			\item[\rm(4)] For the order scaling $\scl_\alpha(t) =e^{-\left(\log\frac1t\right)^\alpha},$ one
			has
			\begin{align*}
				\Gamma_{\scl}(u,b)=\max\{u,1\}.
			\end{align*}

			\item[\rm(5)] Let $c:[0,+\infty)\to[0,+\infty)$ be a strictly increasing homeomorphism, and let $\rho:[0,+\infty)\to[0,+\infty)$
			be continuous with $\rho(0)=0$. For the logarithmically corrected power scaling
			\begin{align*}
				\scl_\alpha(t)
				=
				t^{c(\alpha)}
				\left(1+\log\frac1t\right)^{-\rho(\alpha)},
			\end{align*}
			one has
			\begin{align*}
				\Gamma_{\scl}(u,b)
				=
				c^{-1}\left(c(u)+b\right).
			\end{align*}
		\end{enumerate}
	\end{proposition}

	\begin{proof}
		It is easy to show that they are scalings. We omit the proof here.
		To compute the transfer functions, we first have the following relations
		for the five scalings above:
		\begin{align*}
			e^{(b+\tau)n}
			\frac{\scl_\gamma(e^{-n})}
			{\scl_\alpha(e^{-n})}
			=
			\begin{cases}
				e^{\left(b+\tau+\phi(\alpha)-\phi(\gamma)\right)n},
				&\text{(1)}, \\[2mm]
				e^{(b+\tau)n+(\alpha-\gamma)n^\theta},
				&\text{(2)}, \\[2mm]
				e^{(b+\tau)n}n^{\alpha-\gamma},
				&\text{(3)}, \\[2mm]
				e^{(b+\tau)n+n^\alpha-n^\gamma},
				&\text{(4)}, \\[2mm]
				e^{\left(b+\tau+c(\alpha)-c(\gamma)\right)n}
				(1+n)^{\rho(\alpha)-\rho(\gamma)},
				&\text{(5)}.
			\end{cases}
		\end{align*}

		Then the conclusion for (1)--(3) is obviously.

		For \rm(4), if $\alpha>u$,
		each $\gamma>\max\{\alpha,1\}$ is admissible. Conversely, if $\gamma<\max\{u,1\}$, then either $\gamma<1$ or
		$\gamma<u<\alpha$.
		In the first case,
		\begin{align*}
			(b+\tau)n+n^\alpha-n^\gamma
			\geq (b+\tau)n-n^\gamma
			\longrightarrow+\infty.
		\end{align*} In the second case, since $\gamma<u<\alpha$, we have
		\begin{align*}
		(b+\tau)n+n^\alpha-n^\gamma
		\geq n^\alpha-n^\gamma
		=n^\alpha\left(1-n^{\gamma-\alpha}\right)
		\longrightarrow+\infty.
		\end{align*}
		Thus $\gamma$ is not admissible, and therefore $\Gamma_{\scl}(u,b)=\max\{u,1\}.$

		For \rm(5),  the expression tends to zero if
		$c(\gamma)>c(\alpha)+b+\tau$, and tends to infinity if
		$c(\gamma)<c(\alpha)+b+\tau$.
		Taking the infimum over the admissible parameters gives
		$\Gamma_{\scl}(u,b)=c^{-1}(c(u)+b)$.
		This completes the proof.
	\end{proof}

	It is natural to consider whether the two right-hand inequalities in Theorem~\ref{thm 1.2} can both be attained as equalities, and whether they can both be strict. The following two examples provide affirmative answers to both questions by considering the scaling in Proposition~\ref{prop 4.4}(5).

	For each integer $r\geq2$, let
	$\Sigma_r:=\{0,1,\ldots,r-1\}^{\mathbb N_0}$
	be the one-sided full shift of $r$ symbols, equipped with the standard
	metric \begin{align*}
		d_r(x,y)
		:=
		\begin{cases}
			0, & x=y,\\
			2^{-\min\{j\geq0:x_j\neq y_j\}}, & x\neq y.
		\end{cases}
	\end{align*}
	and let $\sigma_r:\Sigma_r\to\Sigma_r$ be the left shift.
	\begin{example}\label{ex 4.5}
		Let $(X,T)
		=
		\left(
		\Sigma_m\times\Sigma_k,
		\sigma_m\times\sigma_k
		\right)$ and $(Y,S)
		=
		\left(\Sigma_m,\sigma_m\right)$,
		where $m,k\geq2$, and define
		$\pi:X\to Y,\quad
		\pi(x,z):=x.$
		Then $\pi:(X,T)\to(Y,S)$ is a factor map. Moreover,
		\begin{align*}
			\overline h_{\scl}^{C}(T,X)=
			h_{\scl}^{B}(T,X)
			=&
			\Gamma_{\scl}\left(
			\overline h_{\scl}^{C}\left(S,\pi(X)\right),b_\pi
			\right)\\
			=&
			\Gamma_{\scl}\left(
			h_{\scl}^{B}\left(S,\pi(X)\right),b_\pi
			\right)=
			c^{-1}\left(\log m+\log k\right).
		\end{align*}

	\end{example}

	\begin{proof}
		Equip $Y=\Sigma_m$ with $d_Y=d_m$ and
		$X=\Sigma_m\times\Sigma_k$ with the maximum product metric
		\begin{align*}
			d_X\left((x,z),(x',z')\right)
			:=
			\max\left\{
			d_m(x,x'),d_k(z,z')
			\right\}.
		\end{align*}	For every $\alpha>0$, one can observe that
		\begin{align*}
			-\frac1n\log\scl_\alpha(e^{-n})
			=
			c(\alpha)
			+
			\rho(\alpha)\frac{\log(n+1)}{n}
			\longrightarrow c(\alpha).
		\end{align*}
		as $n\to\infty$.
		In particular, for every $\delta>0$ and all sufficiently large $n$,
		\begin{align*}
			e^{-(c(\alpha)+\delta)n}
			\leq
			\scl_\alpha(e^{-n})
			\leq
			e^{-(c(\alpha)-\delta)n},
		\end{align*}
		which shows that
		\begin{align*}
			\overline h_{\scl}^{C}(T,X)=&
			h_{\scl}^{B}(T,X)
			=c^{-1}\left(h_{top}(T,X)\right),\\
			\overline h_{\scl}^{C}(S,Y)
			=&
			h_{\scl}^{B}(S,Y)
			=
			c^{-1}\left(\log m\right).
		\end{align*}

		It is known that $(X,T)$ is conjugate to the full shift on $mk$ symbols.
		Thus $h_{top}(T,X)=\log m+\log k$. We next compute $b_\pi$. Fix $y\in Y$. Since $\pi^{-1}(y)
		=
		\{y\}\times\Sigma_k,$
		we define the map:
		\begin{align*}
			\psi_y:\Sigma_k\to\pi^{-1}(y),
			\quad
			\psi_y(z):=(y,z).
		\end{align*}
		For every $z,z'\in\Sigma_k$ and $n\in\mathbb N$, the maximum
		product metric gives
		\begin{align*}
			d_n^X\left(\psi_y(z),\psi_y(z')\right)
			&=
			\max_{0\leq j<n}
			d_X\left(
			T^j(y,z),T^j(y,z')
			\right)\\
			&=
			\max_{0\leq j<n}
			\max\left\{
			d_m\left(\sigma_m^jy,\sigma_m^jy\right),
			d_k\left(\sigma_k^jz,\sigma_k^jz'\right)
			\right\}\\
			&=
			\max_{0\leq j<n}
			d_k\left(\sigma_k^jz,\sigma_k^jz'\right)
			=
			d_n^{\Sigma_k}(z,z').
		\end{align*}
		In particular, for
		every $\epsilon>0$ and $n\in\mathbb N$, $	r_n\left(\pi^{-1}(y),\epsilon\right)
		=
		r_n\left(\Sigma_k,\epsilon\right).$
		Then
		\begin{align*}
			h_{top}
			\left(T,\pi^{-1}(y)\right)
			=
			h_{top}\left(\sigma_k,\Sigma_k\right)
			=
			\log k.
		\end{align*}
		Since this equality holds for every $y\in Y$, it follows that $	b_\pi
		=
		\log k.$
		Then the conclusion follows  from  Proposition~\ref{prop 4.4}(5).

	\end{proof}

	\begin{example}
		Let $m,k\geq2$, and let $p$ be a point that does not belong to
		$\Sigma_m$ or $\Sigma_k$. Set $X':=\Sigma_m\sqcup\Sigma_k$ and $Y':=\Sigma_m\sqcup\{p\}$.
		Define  $T':X'\to X'$ and $S':Y'\to Y'$ by
		\begin{align*}
			T'|_{\Sigma_m}
			=\sigma_m, \quad
			T'|_{\Sigma_k}
			=\sigma_k,\quad
			S'|_{\Sigma_m}
			=\sigma_m,\quad
			S'(p)=p.
		\end{align*}
		Thus $p$ is a fixed point of $S'$. Define $\pi':X'\to Y'$ by
		\begin{align*}
			\pi'(x)=x,\quad \pi'(z)=p.
		\end{align*}
		for $x\in\Sigma_m$ and  $z\in\Sigma_k$.
		Then
		\begin{align*}
			c^{-1}\left(\max\{\log m,\log k\}\right)
			=
			\overline h_{\scl}^{C}(T',X')
			&<
			\Gamma_{\scl}\left(
			\overline h_{\scl}^{C}(S',Y'),
			b_{\pi'}
			\right)\\
			&=
			c^{-1}\left(\log m+\log k\right),
			\\
			c^{-1}\left(\max\{\log m,\log k\}\right)
			=
			h_{\scl}^{B}(T',X')
			&<
			\Gamma_{\scl}\left(
			h_{\scl}^{B}(S',Y'),
			b_{\pi'}
			\right)\\
			&=
			c^{-1}\left(\log m+\log k\right).
		\end{align*}
	\end{example}

	\begin{proof}
		We first compute $b_{\pi'}$. For every $y\in\Sigma_m$, $(\pi')^{-1}(y)=\{y\}$ and $(\pi')^{-1}(p)=\Sigma_k.$
		Then for each $y\in \Sigma_m$,
		\begin{align*}
			h_{top}\left(T',(\pi')^{-1}(y)\right)=0
			,\quad
			h_{top}\left(T',(\pi')^{-1}(p)\right)
			=\log k.
		\end{align*}
		Hence
		\begin{align*}
			b_{\pi'}
			&=
			\sup_{y\in Y'}
			h_{top}\left(T',(\pi')^{-1}(y)\right)
			=\log k.
		\end{align*}

		By Example  \ref{ex 4.5}, since entropy on a finite disjoint union is the maximum of the
		entropies of its components, we have $	h_{top}(T',X')
		=\max\{\log m,\log k\}$ and $h_{top}(S',Y')=\log m$.
		Therefore,
		\begin{align*}
			\overline h_{\scl}^{C}(T',X')
			=&
			h_{\scl}^{B}(T',X')
			=
			c^{-1}\left(\max\{\log m,\log k\}\right).\\
			\overline h_{\scl}^{C}(S',Y')
			=&
			h_{\scl}^{B}(S',Y')
			=
			c^{-1}\log m.
		\end{align*}
	\end{proof}

	\begin{remark}
		\label{rem 4.6}
	For a general scaling, the following additive upper bounds need not hold:
		\begin{align*}
			\overline h_{\scl}^{C}(S,\pi(E))\leq 	\overline h_{\scl}^{C}(T,E)
			\leq
			\overline h_{\scl}^{C}(S,\pi(E))
			+
			\sup_{y\in Y}
			\overline h_{\scl}^{C}\left(T,\pi^{-1}(y)\right),
			\\
			 h_{\scl}^{B}(S,\pi(E))\leq 	\overline h_{\scl}^{B}(T,E)
			\leq
			 h_{\scl}^{B}(S,\pi(E))
			+
			\sup_{y\in Y}
			 h_{\scl}^{B}\left(T,\pi^{-1}(y)\right).
		\end{align*}
		Indeed, for Example~\ref{ex 4.5}, we take $m=k=2$ and
		$\scl_\alpha(t)
		=
		t^{\sqrt\alpha}
		\left(1+\log\frac1t\right)^{-\alpha}.$
		Then
		\begin{align*}
			\overline h_{\scl}^{C}(T,X)
			=h_{\scl}^{B}(T,X)
			=4(\log 2)^2,\quad
			\overline h_{\scl}^{C}(S,Y)
			=h_{\scl}^{B}(S,Y)
			=(\log 2)^2,
		\end{align*}
		while  for every $y\in Y$,
		\begin{align*}
			\overline h_{\scl}^{C}\left(T,\pi^{-1}(y)\right)
			=
			h_{\scl}^{B}\left(T,\pi^{-1}(y)\right)
			=
			(\log 2)^2.
		\end{align*}
		Thus neither of the corresponding additive factor inequalities holds.
	\end{remark}

	\section{Entropy scales for induced dynamical systems}
	\label{sec 5}
	In this section, we  study (sequence) entropy scales for the induced action on the space of probability
	measures. We prove Theorems~\ref{thm 1.3}
	and~\ref{thm 1.4} with  the corresponding conditions and also provide counterexamples showing that the conclusions may fail
	without such conditions.

	\subsection{Entropy scales under the induced action}
	Let $(X,T)$ be a TDS, and let $\mathscr M(X)$ denote the space of Borel probability measures on $X$
	endowed with the weak$^*$ topology. The induced map is defined by
	\begin{align*}
		T_*\mu=\mu\circ T^{-1},
		\quad \mu\in\mathscr M(X).
	\end{align*}
	Since $\mathscr M(X)$
	is compact in the weak$^*$ topology and $C(X)$ is separable, we may choose a dense sequence
	$\{f_k\}_{k\geq1}$ in the unit ball of $C(X)$ such that
	the metric
	\begin{align}\label{equ 5.1}
		D(\mu,\nu)
		=
		\sum_{k=1}^{\infty}
		\frac{\left|
			\int f_k\,d\mu-\int f_k\,d\nu
			\right|}{2^k}
	\end{align}
	is compatible \cite{Wal82} with the weak$^*$ topology on $\mathscr M(X)$.

	Bauer and Sigmund \cite{BS75} proved that positive topological entropy of the original system
	forces infinite topological entropy of the induced system. Glasner and Weiss \cite{GW95} proved
	that zero topological entropy is preserved by the induced action.

	\begin{theorem}
		\label{thm 5.1}
		Let $(X,T)$ be a TDS. Then the following statements hold.
		\begin{enumerate}
			\item[\rm(1)] $h_{top}(T,X)=0
			\quad\Longleftrightarrow\quad
			h_{top}\left(T_*,\mathscr M(X)\right)=0.$

			\item[\rm(2)] $h_{top}(T,X)>0
			\quad\Longleftrightarrow\quad
			h_{top}\left(T_*,\mathscr M(X)\right)>0\Longleftrightarrow\quad
			h_{top}\left(T_*,\mathscr M(X)\right)=+\infty.$
		\end{enumerate}
	\end{theorem}

	For general scaling families, the conclusions of Theorem~\ref{thm 5.1} need not remain valid without
	additional assumptions. Conditions~\ref{cond 2} and~\ref{cond 3} allow us to recover the corresponding
	zero-level equivalence and positive-to-infinite implication, respectively.

	We first establish a comparison between the covering numbers of the original
	and induced systems. Unlike the classical   approach of, which relies on quasi-factor and measure-theoretic entropy techniques, our proof uses the standard empirical-measure approximation based on Hoeffding's inequality. This approach yields an explicit finite-scale estimate for the covering numbers of the induced system.

	\begin{lemma}
		\label{lem 5.2}
		For every $\epsilon>0$, there exist $\delta_\epsilon>0$, $C_\epsilon>0$,
		and $n_\epsilon\in\mathbb N$ such that
		\begin{align*}
			r_n(\mathscr M(X),D,\epsilon)
			\leq
			r_n(X,d,\delta_\epsilon)^{C_\epsilon\log(n+1)}
		\end{align*}
		holds for all $n\geq n_\epsilon$.
	\end{lemma}

	\begin{proof}
		Fix $\epsilon>0$. Choose $q=q(\epsilon)\in\mathbb N$ and $\eta=\eta(\epsilon)>0$ such that
		$\sum_{k>q}2^{-k}<\frac{\epsilon}{8}$ and $\eta\sum_{k=1}^{q}2^{-k}<\frac{\epsilon}{4}$.
		By the uniform continuity of $f_1,\ldots,f_q$, there exists $\delta_\epsilon>0$ such that
		\begin{align}
			\label{equ 5.3}
			d(x,y)<\delta_\epsilon
			\quad\Longrightarrow\quad
			|f_k(x)-f_k(y)|<\eta,
			\quad 1\leq k\leq q.
		\end{align}

		For $u\in\mathbb R$, denote by $\lceil u\rceil := \min\{p\in\mathbb Z:p\geq u\}$. Fix $n\in\mathbb N$ and $\mu\in\mathscr M(X)$,
		and put
		\begin{align*}
			m
			=
			\left\lceil
			\frac{\log(4qn)}{c_\epsilon}
			\right\rceil,
			\quad
			c_\epsilon:=\frac{\eta^2}{2}.
		\end{align*}
		Consider the product probability space $	(\Omega_m,\mathcal A_m,\mathbb P_m)
		:=
		\left(X^{(m)},\mathcal B(X)^{\otimes m},\mu^{\otimes m}\right).$ For $1\leq\ell\leq m$, define the $\ell$-th coordinate projection:
		\begin{align*}
			\xi_\ell:\Omega_m\longrightarrow X,
			\quad
			\xi_\ell(z_1,\ldots,z_m):=z_\ell.
		\end{align*}
		Then $\xi_1,\ldots,\xi_m$ are independent random variables
		with common distribution $\mu$. We denote by $\mathbb E_m$ the expectation with respect to $\mathbb P_m$ .

		For every $\psi\in \mathcal F_n := \left\{ f_k\circ T^j: 1\leq k\leq q,\ 0\leq j<n \right\},$ one
		has
		\begin{align*}
			\mathbb E_m\left[\psi(\xi_\ell)\right]
			=
			\int_X\psi\,d\mu,
			\quad 1\leq\ell\leq m.
		\end{align*}
		Moreover, $|\psi|\leq1$. Define
		\begin{align*}
			A_\psi
			:=
			\left\{
			\omega\in\Omega_m:
			\left|
			\frac1m\sum_{\ell=1}^{m}
			\psi\left(\xi_\ell(\omega)\right)
			-
			\int_X\psi\,d\mu
			\right|>\eta
			\right\}.
		\end{align*}
		By	Hoeffding's inequality, $\mathbb P_m(A_\psi) \leq 2e^{-c_\epsilon m}.$ Since $|\mathcal F_n|\leq qn$, we have
		\begin{align*}
			\mathbb P_m
			\left(
			\bigcup_{\psi\in\mathcal F_n}A_\psi
			\right)
			\leq
			2qn e^{-c_\epsilon m}
			\leq\frac12.
		\end{align*}
		Choose $\omega=(x_1,\ldots,x_m)
		\in
		\Omega_m\setminus
		\bigcup_{\psi\in\mathcal F_n}A_\psi$,
		one has $\left|
		\frac1m\sum_{\ell=1}^{m}
		\psi(x_\ell)
		-
		\int_X\psi\,d\mu
		\right|
		\leq\eta$
		for each $\psi\in\mathcal F_n.$	Define the empirical measure
		$	\mu_0
		:=
		\frac1m\sum_{\ell=1}^{m}\delta_{x_\ell}.$
		Then for all $1\leq k\leq q$ and $0\leq j<n$,  $	\left|
		\int f_k\circ T^j\,d\mu_0
		-
		\int f_k\circ T^j\,d\mu
		\right|
		\leq\eta$.
		For such $j$, one can deduce that
		\begin{align*}
			D(T_*^j\mu,T_*^j\mu_0)
			&\leq
			\eta\sum_{k=1}^{q}2^{-k}
			+
			2\sum_{k>q}2^{-k}
			<\frac{\epsilon}{2}.
		\end{align*}
		Hence, $D_n(\mu,\mu_0)<\frac{\epsilon}{2}$.

		Let $E_n$ be a  $(n,\delta_\epsilon)$-spanning subset of $X$ with the smallest cardinality. For each $x_\ell$, choose $y_\ell\in E_n$
		such that $d_n(x_\ell,y_\ell)<\delta_\epsilon,$ and set
		$\nu
		:=
		\frac1m\sum_{\ell=1}^{m}\delta_{y_\ell}.$
		For $1\leq k\leq q$ and $0\leq j<n$, by \eqref{equ 5.3},
		\begin{align*}
			\left|
			\int f_k\,d(T_*^j\mu_0)
			-
			\int f_k\,d(T_*^j\nu)
			\right|
			&\leq
			\frac1m
			\sum_{\ell=1}^{m}
			\left|
			f_k(T^jx_\ell)-f_k(T^jy_\ell)
			\right|
			<\eta.
		\end{align*}
		Then  we obtain $D_n(\nu,\mu_0)<\frac{\epsilon}{2}$.
		Thus, for every $\mu\in\mathscr M(X)$, there exists
		\begin{align*}
			\nu
			=
			\frac1m\sum_{\ell=1}^{m}\delta_{y_\ell},
			\quad
			(y_1,\ldots,y_m)\in E_n^m,
		\end{align*}
		such that $D_n(\mu,\nu)<\epsilon$. Hence these empirical measures form an
		$(n,\epsilon)$-spanning set of $\mathscr M(X)$, whose cardinality is at most $|E_n|^m
		=
		r_n(X,d,\delta_\epsilon)^m.$
	Since $q$ and $c_\epsilon$ depend only on $\epsilon$, there
	exist $C_\epsilon>0$ and $n_\epsilon\in\mathbb N$ such that
	$	m\leq C_\epsilon\log(n+1)$
	for every $n\geq n_\epsilon$.	Consequently,
		\begin{align*}
			r_n\left(\mathscr M(X),D,\epsilon\right)
			\leq
			r_n(X,d,\delta_\epsilon)^{
				C_\epsilon\log(n+1)}.
		\end{align*}

	\end{proof}

The following standard construction embeds every finite Cartesian
power of the original system into the induced system \cite[Section~2.5]{BS25}
. For $m\geq2$, write
\begin{align*}
	X^{(m)}
	:=
	\underbrace{X\times\cdots\times X}_{m\text{ times}},\quad
	T^{(m)}
	:=
	\underbrace{T\times\cdots\times T}_{m\text{ times}}.
\end{align*}

\begin{lemma}
	\label{lem 5.3}
	For every integer $m\geq2$, the product system
	$\left(X^{(m)},T^{(m)}\right)$ is topologically conjugate to
	a subsystem of $\left(\mathscr M(X),T_*\right)$.
\end{lemma}

\begin{proof}
	For $1\leq i\leq m$, set $	w_i
	:=
	\frac{2^{i-1}}{2^m-1}.$
	Then $w_i>0$ and $\sum_{i=1}^{m}w_i=1$. Define
	\begin{align*}
		\Phi_m:X^{(m)}
		&\longrightarrow
		\mathscr M(X),\\
		\Phi_m(x_1,\ldots,x_m)
		&:=
		\sum_{i=1}^{m}w_i\delta_{x_i}.
	\end{align*}
	The map $\Phi_m$ is continuous. We claim that it is injective.
	Suppose that
	\begin{align*}
		\Phi_m(x_1,\ldots,x_m)
		=
		\Phi_m(y_1,\ldots,y_m).
	\end{align*}
	For every $z\in X$, evaluating the two measures on the Borel
	set $\{z\}$ gives
	\begin{align*}
		\sum_{\{i:x_i=z\}}2^{i-1}
		=
		\sum_{\{i:y_i=z\}}2^{i-1}.
	\end{align*}
	Since distinct subsets of
	$\{2^0,2^1,\ldots,2^{m-1}\}$ have distinct sums, we obtain
	\begin{align*}
		\{i:x_i=z\}
		=
		\{i:y_i=z\}
	\end{align*}
	for every $z\in X$. Hence $x_i=y_i$ for every
	$1\leq i\leq m$, proving that $\Phi_m$ is injective.
	Moreover,
	\begin{align*}
		T_*\Phi_m(x_1,\ldots,x_m)
		=
		\sum_{i=1}^{m}w_i\delta_{Tx_i}
		=
		\Phi_m(Tx_1,\ldots,Tx_m)
		=
		\Phi_m\circ T^{(m)}(x_1,\ldots,x_m).
	\end{align*}
	Thus $\Phi_m$ is equivariant. Since $X^{(m)}$ is compact
	and $\mathscr M(X)$ is Hausdorff, $\Phi_m$ is a
	homeomorphism from $X^{(m)}$ onto its image. Finally,
	$\Phi_m(X^{(m)})$ is compact and $T_*$-invariant. Therefore,
	$\left(X^{(m)},T^{(m)}\right)$ is topologically conjugate to
	the subsystem $\left(
	\Phi_m(X^{(m)}),
	T_*|_{\Phi_m(X^{(m)})}	\right).$
	\end{proof}

Before proving Theorem~\ref{thm 1.3}, recall that for $\alpha>0$,
\begin{align*}
	\ell_\alpha(n)=-\log\scl_\alpha(e^{-n}),
\end{align*}
and $\ell_\alpha(n)\to+\infty$ as $n\to\infty$.

	\begin{proof}[Proof of Theorem~\ref{thm 1.3}]
		In the proof, we will use  the equivalent definition of entropy scale via spanning sets by Remark~\ref{rem 2.8}.
		We divide the argument into three steps.
		\begin{step1} We show $$\overline h_{\scl}^{C}(T,X)=0
			\Longrightarrow
			\overline h_{\scl}^{C}
			(T_*,\mathscr M(X))=0.$$
		\end{step1}

		Suppose that
		$\overline h_{\scl}^{C}(T,X)=0.$   Then for each $\beta>0$ and every $\eta>0$,
		$$\limsup_{n\to\infty} r_n(X,\eta)\scl_\beta(e^{-n}) =0.$$
		Thus, for all sufficiently large $n$, $\log r_n(X,\eta)\leq \ell_\beta(n).$
		Fix $\alpha_0>0$ in Condition~\ref{cond 2}.
		Fix $0<\alpha<\alpha_0$ and $\epsilon>0$.

		If Condition~\ref{cond 2}{\rm(1)} holds, then there exists  $\beta\in(0,\alpha)$ such that
		\begin{align}
			\label{equ 5.7}
			\lim_{n\to\infty}
			\frac{\ell_\beta(n)\log(n+1)}
			{\ell_\alpha(n)}
			=0.
		\end{align}
		By Lemma~\ref{lem 5.2}, there exist $\delta_\epsilon>0$, $C_\epsilon>0$, and $n_\epsilon\in\mathbb N$
		such that
		\begin{align*}
			r_n(\mathscr M(X),D,\epsilon)
			\leq
			r_n(X,d,\delta_\epsilon)^{C_\epsilon\log(n+1)}
		\end{align*}
		for every $n\geq n_\epsilon$. It shows  $\log r_n\left(\mathscr M(X),\epsilon\right)
		\leq
		C_\epsilon \ell_\beta(n)\log(n+1).$
		Set
		\begin{align*}
			\theta_n
			:=
			C_\epsilon
			\frac{\ell_\beta(n)\log(n+1)}
			{\ell_\alpha(n)}.
		\end{align*}
		By \eqref{equ 5.7}, one has $\theta_n\to0$ since $C_\epsilon$ only depends on $\epsilon$. Hence, for all sufficiently large $n$,
	$\log r_n\left(\mathscr M(X),\epsilon\right)
	\leq
	\theta_n \ell_\alpha(n).$
		It follows that
		\begin{align*}
			r_n\left(\mathscr M(X),\epsilon\right)
			\scl_\alpha(e^{-n})
			=e^{\log r_n\left(\mathscr M(X),\epsilon\right)
				-\ell_\alpha(n)}\leq e^{-(1-\theta_n)\ell_\alpha(n)}.
		\end{align*}
		Since $\theta_n\to0$ and $\ell_\alpha(n)\to+\infty$, it follows that
		\begin{align*}
			\limsup_{n\to\infty}
			r_n\left(\mathscr M(X),\epsilon\right)
			\scl_\alpha(e^{-n})
			=0.
		\end{align*}

	If Condition~\ref{cond 2}{\rm(2)} holds,
	then
	\begin{align*}
		\inf_{\gamma>0}
		\limsup_{n\to\infty}\frac{\ell_\gamma(n)}{n}=0,
		\qquad
		\liminf_{n\to\infty}\frac{\ell_\alpha(n)}{n}>0.
	\end{align*}
		For any fixed  $\delta>0$,  there exists $\gamma>0$ such that $\limsup_{n\to\infty} \frac{\ell_\gamma(n)}{n}<\delta.$
		For every $\eta>0$,
		\begin{align*}
			\limsup_{n\to\infty}
			\frac{\log r_n(X,\eta)}{n}
			\leq
			\limsup_{n\to\infty}
			\frac{\ell_\gamma(n)}{n}
			<\delta.
		\end{align*}
		Letting $\eta\to0$ and then $\delta\to0$ shows $h_{top}(T,X)=0$. By Theorem~\ref{thm 5.1},
		$h_{top}
		\left(T_*,\mathscr M(X)\right)=0.$
		Fix $\epsilon>0$.
		There exist $c_\alpha>0$ and $N_\alpha\in\mathbb N$ such that $\ell_\alpha(n)\geq c_\alpha n$
		for every $n\geq N_\alpha$. Hence
		\begin{align*}
			0\leq
			\limsup_{n\to\infty}
			\frac{
				\log r_n\left(\mathscr M(X),\epsilon\right)}
			{\ell_\alpha(n)}\leq
			\frac{1}{c_\alpha}
			\limsup_{n\to\infty}
			\frac{
				\log r_n\left(\mathscr M(X),\epsilon\right)}
			{n}=0.
		\end{align*}
		It follows that  for all sufficiently large $n$, one has $\log r_n\left(\mathscr M(X),\epsilon\right) \leq \frac{1}{2}\ell_\alpha(n).$
		Therefore,
		\begin{align*}
			r_n\left(\mathscr M(X),\epsilon\right)
			\scl_\alpha(e^{-n})
			=e^{\log r_n\left(\mathscr M(X),\epsilon\right)
				-\ell_\alpha(n)}
			\leq e^{-\frac{1}{2}\ell_\alpha(n)},
		\end{align*}
		which shows $	\limsup_{n\to\infty}
		r_n\left(\mathscr M(X),\epsilon\right)
		\scl_\alpha(e^{-n})
		=0.$

		Thus, for every $0<\alpha<\alpha_0$ and every $\epsilon>0$,
		\begin{align*}
			\limsup_{n\to\infty}
			r_n\left(\mathscr M(X),\epsilon\right)
			\scl_\alpha(e^{-n})=0.
		\end{align*}
		Letting $\epsilon\to0$ gives
		$\overline h_{\scl}^{C}(T_*,\mathscr M(X))\leq\alpha$.
		Letting $\alpha\to0$ implies the desired result.

		\begin{step2} We show $$	\overline h_{\scl}^{C}
			(T_*,\mathscr M(X))=0
			\Longrightarrow \overline h_{\scl}^{C}(T,X)=0.$$
		\end{step2}
		Suppose that $\overline h_{\scl}^{C} (T_*,\mathscr M(X))=0.$ Consider the Dirac embedding
		\begin{align*}
			\Delta:X&\longrightarrow\mathscr M(X),\\
			x&\longmapsto\delta_x.
		\end{align*}
		It is a homeomorphism from $X$ onto the compact subset $\Delta(X)\subset\mathscr M(X)$, and it
		satisfies $	\Delta\circ T=T_*\circ\Delta.$
		By virtue of uniform continuity of $\Delta^{-1}$, for every $\eta>0$, there exists $\epsilon>0$
		such that for all $x,y\in X$,
		\begin{align}
			\label{equ 5.10}
			D(\delta_x,\delta_y)<\epsilon
			\quad\Longrightarrow\quad
			d(x,y)<\eta.
		\end{align}

		We claim that if $F\subset X$ is an $(n,\eta)$-separated set, then
		$\Delta(F)
		:=
		\left\{\delta_x:x\in F\right\}$
		is an $(n,\epsilon)$-separated subset of $\mathscr M(X)$.

		Indeed, we select two distinct points $x,y\in F$. Then there exists $0\leq j<n$ such that $d\left(T^jx,T^jy\right)
		\geq
		\eta.$
		By  \eqref{equ 5.10}, it shows $	D\left(\delta_{T^jx},\delta_{T^jy}\right)
		\geq
		\epsilon.$
		Since $T_*^j\delta_x=\delta_{T^jx}$ and $T_*^j\delta_y=\delta_{T^jy}$,
		we obtain
		\begin{align*}
			D_n\left(\delta_x,\delta_y\right)
			=
			\max_{0\leq i<n}
			D\left(T_*^i\delta_x,T_*^i\delta_y\right)
			\geq
			D\left(T_*^j\delta_x,T_*^j\delta_y\right)
			=
			D\left(\delta_{T^jx},\delta_{T^jy}\right)
			\geq
			\epsilon.
		\end{align*}
		Therefore, $\Delta(F)$ is $(n,\epsilon)$-separated. Moreover,
		$	s_n(X,\eta)
		\leq
		s_n\left(\mathscr M(X),\epsilon\right).$

		Since $\overline h_{\scl}^{C}
		\left(T_*,\mathscr M(X)\right)
		=
		0,$ for any $\alpha>0$ and $\epsilon>0$, one can show that
		\begin{align*}
			0\leq 	\limsup_{n\to\infty}
			s_n\left(X,\eta\right)
			\scl_\alpha(e^{-n})
			\leq 	\limsup_{n\to\infty}
			s_n\left(\mathscr M(X),\epsilon\right)
			\scl_\alpha(e^{-n})=
			0.
		\end{align*}
		Letting $\eta\to 0$ indicates that
		$\overline h_{\scl}^{C}(T,X)=0.$

	\begin{step3}
		Under Condition~\ref{cond 3}, we prove that
		\begin{align*}
			\overline h_{\scl}^{C}(T,X)>0
			\quad\Longrightarrow\quad
			\overline h_{\scl}^{C}
			\left(T_*,\mathscr M(X)\right)=+\infty.
		\end{align*}
	\end{step3}

	Assume $\overline h_{\scl}^{C}(T,X)>0$.
	Fix $\gamma>0$ and choose $0<\alpha<
	\min\left\{\gamma,\overline h_{\scl}^{C}(T,X)\right\}.$
	Then $\overline\Lambda_{\scl_\alpha}(T,X)=+\infty$.
	By Condition~\ref{cond 3}, at least one of the following
	two cases holds.

	If Condition~\ref{cond 3}{\rm(1)} holds:
	\begin{align*}
		\limsup_{n\to\infty}\frac{\ell_\gamma(n)}{\ell_\alpha(n)}<\infty.
	\end{align*}
	We choose an integer
	\begin{align*}
		m>\max\left\{
		1,\limsup_{n\to\infty}
		\frac{\ell_\gamma(n)}{\ell_\alpha(n)}
		\right\}.
	\end{align*}
Then  for all sufficiently large $n$, one has
	\begin{align*}
		\ell_\gamma(n)\leq m\ell_\alpha(n),
		\qquad
		\scl_\gamma(e^{-n})\geq\scl_\alpha(e^{-n})^m.
	\end{align*}
	Under the maximum product metric $d^{(m)}$, the product
	of an $(n,\epsilon)$-separated set with itself $m$ times
	is $(n,\epsilon)$-separated. Hence
	\begin{align*}
		s_n\left(X^{(m)},d^{(m)},\epsilon\right)
		\scl_\gamma(e^{-n})
		\geq
		\left(s_n(X,d,\epsilon)\scl_\alpha(e^{-n})\right)^m.
	\end{align*}
	Taking the limsup as $n\to\infty$ and then letting
	$\epsilon\to0$ gives
	$	\overline\Lambda_{\scl_\gamma}
	\left(T^{(m)},X^{(m)}\right)=+\infty.$
	Lemma~\ref{lem 5.3} and Proposition~\ref{prop 2.10} imply
	\begin{align*}
		\overline h_{\scl}^{C}
		\left(T_*,\mathscr M(X)\right)
		\geq
		\overline h_{\scl}^{C}
		\left(T^{(m)},X^{(m)}\right)
		\geq\gamma.
	\end{align*}

	If Condition~\ref{cond 3}{\rm(2)} holds, then there exists $a,b\in (0,\infty)$ such that
	\begin{align*}
		0<a<\liminf_{n\to\infty}\frac{\ell_\alpha(n)}n\leq
		\liminf_{n\to\infty}\frac{\ell_\gamma(n)}n<b<+\infty.
	\end{align*}
	Thus $\ell_\alpha(n)\geq an$ for all sufficiently large $n$.
	If $h_{top}(T,X)=0$, then for every $\epsilon>0$,
	\begin{align*}
		s_n(X,d,\epsilon)\scl_\alpha(e^{-n})
		\leq s_n(X,d,\epsilon)e^{-an}
		\longrightarrow0,
	\end{align*}
	contradicting $\overline\Lambda_{\scl_\alpha}(T,X)=+\infty$.
	Therefore $h_{top}(T,X)>0$, and Theorem~\ref{thm 5.1} gives
	\begin{align*}
		h_{top}\left(T_*,\mathscr M(X)\right)=+\infty.
	\end{align*}
	Choose an increasing sequence $\{n_j\}$ such that
	$\ell_\gamma(n_j)\leq bn_j$.
	Choose a finite open cover $\mathcal U$ of $\mathscr M(X)$
	with $h_{top}(T_*,\mathcal U)>b+1$, and choose $\epsilon>0$
	such that every $D$-ball of radius $\epsilon$ is contained
	in an element of $\mathcal U$.
	Since a maximal $(n,\epsilon)$-separated set is spanning,
	\begin{align*}
		s_n\left(\mathscr M(X),D,\epsilon\right)
		\geq
		N\left(\bigvee_{i=0}^{n-1}(T_*)^{-i}\mathcal U\right)
		\geq e^{nh_{top}(T_*,\mathcal U)}
		\geq e^{(b+1)n},
	\end{align*}
	where $N(\mathcal U)$ denotes the minimum cardinality of a subcover.
	Consequently,
	\begin{align*}
		s_{n_j}\left(\mathscr M(X),D,\epsilon\right)
		\scl_\gamma(e^{-n_j})
		\geq
		e^{(b+1)n_j-\ell_\gamma(n_j)}
		\geq e^{n_j}
		\longrightarrow+\infty.
	\end{align*}
	Thus
	$\overline h_{\scl}^{C}(T_*,\mathscr M(X))\geq\gamma$.

	Since $\gamma>0$ was arbitrary, we conclude that $\overline h_{\scl}^{C}
	\left(T_*,\mathscr M(X)\right)=+\infty.$
	\end{proof}

	We then immediately obtain the following corollary.
	\begin{corollary}
		Let $\scl$ be a scaling satisfying Conditions~\ref{cond 2} and~\ref{cond 3}. Then
		\begin{align*}
			\overline h_{\scl}^{C}(T,X)>0\Longleftrightarrow
			\overline h_{\scl}^{C}
			\left(T_*,\mathscr M(X)\right)>0
			\Longleftrightarrow
			\overline h_{\scl}^{C}
			\left(T_*,\mathscr M(X)\right)=+\infty.
		\end{align*}
	\end{corollary}

	For the power scaling $\scl_\alpha(t)=t^\alpha$, one can easily check  that  Conditions~\ref{cond 2} and~\ref{cond 3} are satisfied, and
	Theorem~\ref{thm 1.3} recovers the classical induced-system
	dichotomy in Theorem~\ref{thm 5.1}.

	The following two examples show that the corresponding conclusions
	may fail when Conditions~\ref{cond 2} and~\ref{cond 3},
	respectively, are not assumed.
	The first example shows that the zero upper capacity
	entropy scale may not be preserved without Condition~\ref{cond 2}.
	\begin{example}\label{ex 5.7}
		Let $X_0=\mathbb Z\cup\{\infty\}$ be the one-point compactification of $\mathbb Z$, and define
		\begin{align*}
			T(k)=k+1,\quad T(\infty)=\infty .
		\end{align*}
		Consider the scaling $\{\scl_\alpha\}_{\alpha\geq0}$ defined by $\scl_0(t)=1$ and
		\begin{align*}
			\scl_\alpha(t)
			=
			\begin{cases}
				|\log t|^{-(1+\alpha)}, & 0<t<e^{-1},\\
				1,                     & e^{-1}\leq t<1.
			\end{cases}
		\end{align*}
		for $\alpha>0$, which does not satisfy Condition \ref{cond 2}.
		Then
		\begin{align*}
0= \overline{h}_{\scl}^{C}(T,X_0)\neq \overline{h}_{\scl}^{C}\left(T_*,\mathscr M(X_0)\right)=+\infty.
		\end{align*}
	\end{example}

	\begin{proof}
		Fix $\alpha>0$. For sufficiently large $n$, one has $\ell_\alpha(n)=(1+\alpha)\log n$. Then Condition \ref{cond 2} fails.

		We next show that
		\begin{align*}
			0=\overline h_{\scl}^{C}(T,X_0)\neq \overline{h}_{\scl}^{C}\left(T_*,\mathscr M(X_0)\right)=+\infty.
		\end{align*}
		Since $\{0\}$ is open in $X_0$, there exists $\epsilon_0>0$ such that $B_d(0,\epsilon_0)=\{0\}.$ Fix $0<\epsilon<\epsilon_0$ and put
		$	E_n:=\{0,-1,\ldots,-(n-1)\}.$
		For distinct $x=-i$ and $y=-j$ in $E_n$, assume that $i<j$. Then
		$T^i(x)=0$ and $T^i(y)=i-j\neq0.$
		Consequently,
		\begin{align*}
			d_n(x,y) \geq d(T^i(x),T^i(y)) =d(0,i-j) \geq\epsilon_0 >\epsilon.
		\end{align*}
		Thus $E_n$ is $(n,\epsilon)$-separated
		and $s_n(X_0,\epsilon)\geq n$.
		For the reverse bound, compatibility of $d$ with the topology of $X_0$ gives a finite set
		$	F_\epsilon
		=\{a_1,\ldots,a_{r(\epsilon)}\}
		\subset\mathbb Z$
		such that $\operatorname{diam}_d(X_0\setminus F_\epsilon)<\epsilon$. Let $E\subset X_0$ be an
		arbitrary $(n,\epsilon)$-separated set and denote by
		\begin{align*}
			E_0
			&:=
			\{x\in E:T^j(x)\notin F_\epsilon
			\text{ for all }0\leq j<n\},
			\\
			E_1&:=E\setminus E_0.
		\end{align*}
		If $x,y\in E_0$, then $T^j(x),T^j(y)\in X_0\setminus F_\epsilon$ for every $0\leq j<n$, and hence $d_n(x,y)<\epsilon$.
		Therefore $|E_0|\leq1$. For every $x\in E_1$, there exist $a_i\in F_\epsilon$ and $0\leq j<n$ such
		that $T^j(x)=a_i$. Since $T(k)=k+1$ on $\mathbb Z$, it follows that $x=a_i-j$. Thus
		\begin{align*}
			E_1
			\subset
			\{a_i-j:1\leq i\leq r(\epsilon),\ 0\leq j<n\},
		\end{align*}
		so $|E_1|\leq r(\epsilon)n$. Taking the supremum over all such $E$ gives $n \leq s_n(X_0,\epsilon) \leq r(\epsilon)n+1.$
		It follows that, for every $\alpha>0$,
		\begin{align*}
			s_n(X_0,\epsilon)\scl_\alpha(e^{-n})
			\leq
			\left(r(\epsilon)n+1\right)n^{-(1+\alpha)}
			=
			r(\epsilon)n^{-\alpha}+n^{-(1+\alpha)}.
		\end{align*}
		Therefore, $\overline h_{\scl}^{C}(T,X_0)=0$.

		It remains to prove
		$\overline h_{\scl}^{C}\left(T_*,\mathscr M(X_0)\right)=+\infty$.
		Since $\{0\}$ is clopen in $X_0$, the function $\varphi:=\mathbf 1_{\{0\}}$ is continuous.
		Choose the sequence defining the metric $D$ in \eqref{equ 5.1} so that $f_1=\varphi$. Then
		$	D(\mu,\nu)
		\geq
		\frac12
		\left|
		\int\varphi\,d\mu-
		\int\varphi\,d\nu
		\right|.$
		Fix $m\in\mathbb N$. For every $A\subset\{0,1,\ldots,n-1\}$ with $|A|=m$, define the measure by
		\begin{align*}
			\mu_A:=\frac1m\sum_{j\in A}\delta_{-j},
		\end{align*}
		and set $\mathcal F_n(m)
		:=
		\left\{
		\mu_A:
		A\subset\{0,\ldots,n-1\},\ |A|=m
		\right\}.$
		We claim that $\mathcal F_n(m)$ is $(n,\frac{1}{2m})$-separated. Indeed, take distinct
		$\mu_A,\mu_B\in\mathcal F_n(m)$. Since $A\neq B$ and $|A|=|B|$, choose
		$j\in A\setminus B$. Then
		\begin{align*}
			\int\varphi\,d(T_*^j\mu_A)=\frac1m,
			\quad
			\int\varphi\,d(T_*^j\mu_B)=0.
		\end{align*}
		Hence, $D_n(\mu_A,\mu_B)
		=
		\max_{0\leq k<n}D(T_*^k\mu_A,T_*^k\mu_B)
		\geq
		D(T_*^j\mu_A,T_*^j\mu_B)
		\geq\frac1{2m}.$
		Therefore,
		\begin{align*}
			s_n\left(\mathscr M(X_0),\frac1{2m}\right)
			\geq
			|\mathcal F_n(m)|
			=
			\binom nm.
		\end{align*}
		For fixed $m$, there exists $c_m>0$ such that $\binom nm\geq c_mn^m$ for all sufficiently large $n$.
		Consequently,
		\begin{align*}
			s_n\left(\mathscr M(X_0),\frac1{2m}\right)
			\scl_\alpha(e^{-n})
			\geq
			c_mn^{m-1-\alpha}.
		\end{align*}
		Given $\alpha>0$, choose an integer $m>1+\alpha$. Then the right-hand side tends to $+\infty$.
		Since $\alpha>0$ is arbitrary, it follows that
		$\overline h_{\scl}^{C}\left(T_*,\mathscr M(X_0)\right)=+\infty$.
	\end{proof}

	The next example shows that the upper capacity entropy scale of the induced action may remain finite	even when that of the original system is positive.
	\begin{example}\label{ex 5.8}
		Consider
		the order scaling $\scl_\alpha(t)
		=e^{-\left(\log\frac1t\right)^\alpha}$ and the one-sided full shift  $(\Sigma_m,\sigma)$ with $m\geq2$.
		Then one has
		\begin{align*}
			\overline h_{\scl}^{C}(\sigma,\Sigma_m)
			=
			\overline h_{\scl}^{C}
			\left(\sigma_*,\mathscr M(\Sigma_m)\right)
			=1.
		\end{align*}
	\end{example}

	\begin{proof}
	Notice that $\ell_\alpha(n)=n^\alpha$.
	Taking $\alpha=1$ and $\gamma=2$, we obtain
	\begin{align*}
		\frac{\ell_2(n)}{\ell_1(n)}=
		\frac{\ell_2(n)}{n}=n\longrightarrow+\infty.
	\end{align*}
	Thus  Condition~\ref{cond 3} fails.

		We next compute the corresponding entropy scales. Fix $\epsilon>0$.
		Choose finite open covers $\mathcal U_\epsilon$ of $\Sigma_m$ and
		$\mathcal V_\epsilon$ of $\mathscr M(\Sigma_m)$ such that their elements
		have diameter less than $\epsilon$.
		Choosing one point from each nonempty element of $	\bigvee_{j=0}^{n-1}\sigma^{-j}\mathcal U_\epsilon$ and $\bigvee_{j=0}^{n-1}\sigma_*^{-j}\mathcal V_\epsilon$. Hence
		\begin{align*}
			r_n(\Sigma_m,\epsilon)
			\leq |\mathcal U_\epsilon|^n,\quad
			r_n\left(\mathscr M(\Sigma_m),\epsilon\right)
			\leq |\mathcal V_\epsilon|^n.
		\end{align*}
		Therefore, for every $\alpha>1$,
		\begin{align*}
			s_n(\Sigma_m,2\epsilon)\scl_\alpha(e^{-n})
			&\leq
			|\mathcal U_\epsilon|^n e^{-n^\alpha}
			=
			e^{n\log |\mathcal U_\epsilon|-n^\alpha}
			\longrightarrow0,
		\end{align*}
		and
		\begin{align*}
			s_n\left(\mathscr M(\Sigma_m),2\epsilon\right)
			\scl_\alpha(e^{-n})
			&\leq
			|\mathcal V_\epsilon|^n e^{-n^\alpha}
			=
			e^{n\log |\mathcal V_\epsilon|-n^\alpha}
			\longrightarrow0.
		\end{align*}
		Consequently,
		\begin{align*}
			\overline h_{\scl}^{C}(\sigma,\Sigma_m)
			\leq1,\quad
			\overline h_{\scl}^{C}
			\left(\sigma_*,\mathscr M(\Sigma_m)\right)
			\leq1.
		\end{align*}

		Choose $\epsilon_0>0$ such that $\Sigma_m$ contains an
		$(n,\epsilon_0)$-separated set of cardinality $m^n$ for every
		$n\in\mathbb N$. If $0<\alpha<1$, then
		\begin{align*}
			s_n(\Sigma_m,\epsilon_0)\scl_\alpha(e^{-n})
			\geq
			m^n e^{-n^\alpha}
			=
			e^{n\log m-n^\alpha}
			\longrightarrow+\infty.
		\end{align*}
		Thus $	\overline h_{\scl}^{C}(\sigma,\Sigma_m)\geq1.$
		Finally, the equivariant Dirac embedding identifies
		$(\Sigma_m,\sigma)$ with a subsystem of
		$(\mathscr M(\Sigma_m),\sigma_*)$. By
		Proposition~\ref{prop 2.10},
		\begin{align*}
			\overline h_{\scl}^{C}
			\left(\sigma_*,\mathscr M(\Sigma_m)\right)
			\geq
			\overline h_{\scl}^{C}(\sigma,\Sigma_m)
			\geq1.
		\end{align*}
		Then we finish the proof.
	\end{proof}

	\subsection{Sequence entropy scales under the induced action}
	\label{subsec 5.3}

	We now consider entropy scales along an increasing sequence.
	Fix an increasing sequence $\mathsf S=\{t_1<t_2<\cdots\}\subset\mathbb Z_+.$ For $n\in\mathbb N$,
	define
	\begin{align*}
		d_n^{\mathsf S}(x,y)
		:=
		\max_{1\leq j\leq n}
		d\left(T^{t_j}x,T^{t_j}y\right).
	\end{align*}
	A subset $E\subset X$ is called $(\mathsf S,n,\epsilon)$-separated if $d_n^{\mathsf S}(x,y)\geq\epsilon$
	for  distinct points $x,y\in E$. A subset $F\subset X$ is called $(\mathsf S,n,\epsilon)$-spanning
	if  for every $x\in X$, there exists $y\in F$ such that $d_n^{\mathsf S}(x,y)<\epsilon.$ Denote the
	maximal cardinality of an $(\mathsf S,n,\epsilon)$-separated set and the minimal cardinality of an $(\mathsf S,n,\epsilon)$-spanning
	set by $s_n^{\mathsf S}(X,d,\epsilon)$ and $r_n^{\mathsf S}(X,d,\epsilon)$, respectively. Similarly, it holds that
	\begin{align}
		\label{equ 5.14}
		s_n^{\mathsf S}(X,d,2\epsilon)
		\leq
		r_n^{\mathsf S}(X,d,\epsilon)
		\leq
		s_n^{\mathsf S}(X,d,\epsilon).
	\end{align}

	For every $\alpha>0$, set
	\begin{align*}
		\overline\Lambda_{\scl_\alpha}^{\mathsf S}
		(T,X,d,\epsilon):=&\limsup_{n\to\infty}
		s_n^{\mathsf S}(X,d,\epsilon)\scl_\alpha(e^{-n}),\\
		\overline\Lambda_{\scl_\alpha}^{\mathsf S}(T,X,d)
		:=&
		\lim_{\epsilon\to 0}
		\overline\Lambda_{\scl_\alpha}^{\mathsf S}
		(T,X,d,\epsilon).
	\end{align*}
	The scaling property implies that there is a critical value at which $\overline\Lambda_{\scl_\alpha}^{\mathsf S}(T,X,d)$
	changes from $+\infty$ to $0$.

	\begin{definition}
		\label{def 5.9}
		The \emph{upper capacity sequence entropy scale} of $T$ along $\mathsf S$ is defined by
		\begin{align*}
			\overline h_{\scl,\mathsf S}^{C}(T,X)
			:=&
			\inf\left\{
			\alpha>0:
			\overline\Lambda_{\scl_\alpha}^{\mathsf S}(T,X,d)=0
			\right\}\\
			=&
			\sup\left\{
			\alpha>0:
			\overline\Lambda_{\scl_\alpha}^{\mathsf S}(T,X,d)=+\infty
			\right\}.
		\end{align*}
	\end{definition}

	The value  is also independent of the choice of a compatible metric. By
	\eqref{equ 5.14}, it has no effect if we replace separated sets with spanning sets.

	\begin{remark}
		\label{rem 5.10}
		For the power scaling $\scl_\alpha(\epsilon)=\epsilon^\alpha$, one has $\overline h_{\scl,\mathsf S}^{C}(T,X)
		=
		h_{top}^{\mathsf S}(T,X),$
		where $h_{top}^{\mathsf S}(T,X)$ denotes the topological sequence entropy \cite{Goo74} along $\mathsf S$.
		If $t_j=j-1$, then $\overline h_{\scl,\mathsf S}^{C}(T,X)$ reduces to the upper capacity entropy
		scale.
	\end{remark}

	Similar with Lemma \ref{lem 5.2}, we can also    give the covering estimate for the sequence entropy scales. Define
	\begin{align*}
		D_n^{\mathsf S}(\mu,\nu)
		:=
		\max_{1\leq j\leq n}
		D\left(T_*^{t_j}\mu,T_*^{t_j}\nu\right).
	\end{align*}

	\begin{lemma}
		\label{lem 5.11}
		For every $\epsilon>0$, there exist $\delta_\epsilon>0$, $C_\epsilon>0$, and $n_\epsilon\in\mathbb N$
		such that
		\begin{align*}
			r_n^{\mathsf S}
			\left(\mathscr M(X),D,\epsilon\right)
			\leq
			\left(
			r_n^{\mathsf S}(X,d,\delta_\epsilon)
			\right)^{C_\epsilon\log(n+1)}
		\end{align*}
		for every $n\geq n_\epsilon$.
	\end{lemma}

	\begin{proof}
		Fix $\epsilon>0$, and choose $\delta_\epsilon>0$ as in the proof of
		Lemma~\ref{lem 5.2}. For $n\in\mathbb N$, replace the family $\left\{
		f_k\circ T^j:
		1\leq k\leq q,\ 0\leq j<n
		\right\}$ by
		\begin{align*}
			\left\{
			f_k\circ T^{t_j}:
			1\leq k\leq q,\ 1\leq j\leq n
			\right\}.
		\end{align*}
		This new family also contains at most $qn$ functions. Therefore, the
		same application of Hoeffding's inequality shows that every
		$\mu\in\mathscr M(X)$ can be approximated, at all times
		$t_1,\ldots,t_n$, by an empirical measure $	\mu_0
		=
		\frac{1}{m}\sum_{\ell=1}^{m}\delta_{x_\ell}$ and $m\leq C_\epsilon\log(n+1).$

		Let $E_n$ be a minimal $(\mathsf S,n,\delta_\epsilon)$-spanning set
		of $X$. For each $x_\ell$, choose $y_\ell\in E_n$ such that
		\begin{align*}
			\max_{1\leq j\leq n}
			d\left(T^{t_j}x_\ell,T^{t_j}y_\ell\right)
			<
			\delta_\epsilon,
		\end{align*}
		and define $	\nu
		=
		\frac{1}{m}\sum_{\ell=1}^{m}\delta_{y_\ell}.$
		Similar with  Lemma~\ref{lem 5.2}, one obtains
		\begin{align*}
			\max_{1\leq j\leq n}
			D\left(
			T_*^{t_j}\mu,T_*^{t_j}\nu
			\right)
			<
			\epsilon.
		\end{align*}
		Thus the empirical measures supported on at most
		$C_\epsilon\log(n+1)$ points of $E_n$ form an
		$(\mathsf S,n,\epsilon)$-spanning set of $\mathscr M(X)$. Hence,
		\begin{align*}
			r_{n}^{\mathsf S}
			\left(
			\mathscr M(X),D,\epsilon
			\right)
			\leq
			|E_n|^m
			\leq
			r_{n}^{\mathsf S}
			\left(
			X,d,\delta_\epsilon
			\right)^{C_\epsilon\log(n+1)}.
		\end{align*}
	\end{proof}

	We now prove Theorem~\ref{thm 1.4}.
\begin{proof}[Proof of Theorem~\ref{thm 1.4}]
	We divide the proof into two steps.
\begin{step1}
We prove that if $\scl$ satisfies
	Condition~\ref{cond 2}, then
	\begin{align*}
		\overline h_{\scl,\mathsf S}^{C}(T,X)=0
		\quad\Longleftrightarrow\quad
		\overline h_{\scl,\mathsf S}^{C}
		\left(T_*,\mathscr M(X)\right)=0.
	\end{align*}
\end{step1}

Similar with the   Step~2 of the proof of Theorem~\ref{thm 1.3}, we are esaily to show $$\overline h_{\scl,\mathsf S}^{C}
\left(T_*,\mathscr M(X)\right)=0
\Longrightarrow
\overline h_{\scl,\mathsf S}^{C}(T,X)=0.$$
For the converse implication, we
 assume
	$\overline h_{\scl,\mathsf S}^{C}(T,X)=0$.
	For every $\beta,\eta>0$ and all sufficiently large $n$, the inequality holds:
	\begin{align*}
		\log r_n^{\mathsf S}(X,d,\eta)
		\leq\ell_\beta(n).
	\end{align*}
	Let $\alpha_0>0$ be given by Condition~\ref{cond 2}.
	Fix $0<\alpha<\alpha_0$ and $\epsilon>0$.

	If Condition~\ref{cond 2}{\rm(1)} holds for $\alpha$,
	choose $\beta\in(0,\alpha)$ such that
	\begin{align*}
		\lim_{n\to\infty}
		\frac{\ell_\beta(n)\log(n+1)}{\ell_\alpha(n)}=0.
	\end{align*}
	By Lemma~\ref{lem 5.11},
	\begin{align*}
		0\leq
		\frac{\log r_n^{\mathsf S}
			\left(\mathscr M(X),D,\epsilon\right)}{\ell_\alpha(n)}
		\leq
		C_\epsilon
		\frac{\ell_\beta(n)\log(n+1)}{\ell_\alpha(n)}
		\longrightarrow0.
	\end{align*}
	Hence, for all sufficiently large $n$,
	\begin{align*}
		r_n^{\mathsf S}
		\left(\mathscr M(X),D,\epsilon\right)
		\scl_\alpha(e^{-n})
		\leq e^{-\ell_\alpha(n)/2}
		\longrightarrow0.
	\end{align*}

	Otherwise, Condition~\ref{cond 2}{\rm(2)} gives
	\begin{align*}
		\inf_{\gamma>0}
		\limsup_{n\to\infty}\frac{\ell_\gamma(n)}n=0,
		\qquad
		\liminf_{n\to\infty}\frac{\ell_\alpha(n)}n>0.
	\end{align*}
	For every $\eta>0$, one has
	\begin{align*}
		0\leq
		\limsup_{n\to\infty}
		\frac{\log r_n^{\mathsf S}(X,d,\eta)}n
		\leq
		\inf_{\gamma>0}
		\limsup_{n\to\infty}\frac{\ell_\gamma(n)}n
		=0.
	\end{align*}
	Letting $\eta\to0$ gives $h_{top}^{\mathsf S}(T,X)=0$.
Applying \cite[Theorem~5.3]{Shao26} to the constant
sequence of maps $f_n=T$, we obtain
$h_{top}^{\mathsf S}
\left(T_*,\mathscr M(X)\right)=0.$
	Choose a finite constant $c_\alpha$ such that
$	0<c_\alpha<
\liminf_{n\to\infty}\frac{\ell_\alpha(n)}n.$
Then for all sufficiently large $n$,
	\begin{align*}
		\log r_n^{\mathsf S}
		\left(\mathscr M(X),D,\epsilon\right)
		\leq\frac{c_\alpha}{2}n,
		\qquad
		\ell_\alpha(n)\geq c_\alpha n.
	\end{align*}
	Therefore,
	\begin{align*}
		r_n^{\mathsf S}
		\left(\mathscr M(X),D,\epsilon\right)
		\scl_\alpha(e^{-n})
		\leq e^{-c_\alpha n/2}
		\longrightarrow0.
	\end{align*}

	Hence, we obtain
	$\overline h_{\scl,\mathsf S}^{C}
	(T_*,\mathscr M(X))\leq\alpha$.
	Letting $\alpha\to 0$ show
$\overline h_{\scl,\mathsf S}^{C}
\left(T_*,\mathscr M(X)\right)=0.$

	\begin{step2} We prove that
		if $\limsup_{n\to\infty}
		\frac{\ell_\gamma(n)}{\ell_\alpha(n)}
		<+\infty$ for every $0<\alpha<\gamma,$
		then
		\begin{align*}
			\overline h_{\scl,\mathsf S}^{C}(T,X)>0
			\quad\Longrightarrow\quad
			\overline h_{\scl,\mathsf S}^{C}
			\left(T_*,\mathscr M(X)\right)=+\infty.
		\end{align*}
	\end{step2}
	Assume $\overline h_{\scl,\mathsf S}^{C}(T,X)>0$.
	Fix $\gamma>0$ and choose
	$0<\alpha<
	\min\left\{
	\gamma,\overline h_{\scl,\mathsf S}^{C}(T,X)
	\right\}.$
	Then
	\begin{align*}
		\lim_{\epsilon\to0}
		\limsup_{n\to\infty}
		s_n^{\mathsf S}(X,d,\epsilon)
		\scl_\alpha(e^{-n})=+\infty.
	\end{align*}
	By the hypothesis of {\rm(2)}, choose an integer
	$m>\max\left\{
	1,\limsup_{n\to\infty}
	\frac{\ell_\gamma(n)}{\ell_\alpha(n)}
	\right\}.$
	For all sufficiently large $n$,
	\begin{align*}
		\ell_\gamma(n)\leq m\ell_\alpha(n),
		\qquad
		\scl_\gamma(e^{-n})\geq\scl_\alpha(e^{-n})^m.
	\end{align*}

	Equip $X^{(m)}$ with the maximum product metric $d^{(m)}$.
	If $E\subset X$ is $(\mathsf S,n,\epsilon)$-separated,
	then $E^{(m)}$ is $(\mathsf S,n,\epsilon)$-separated.
	Thus
	$	s_n^{\mathsf S}\left(X^{(m)},d^{(m)},\epsilon\right)
	\geq s_n^{\mathsf S}(X,d,\epsilon)^m.$
	Consequently, for all sufficiently large $n$,
	\begin{align*}
		s_n^{\mathsf S}\left(X^{(m)},d^{(m)},\epsilon\right)
		\scl_\gamma(e^{-n})
		\geq
		\left(
		s_n^{\mathsf S}(X,d,\epsilon)\scl_\alpha(e^{-n})
		\right)^m.
	\end{align*}
	Taking the limsup as $n\to\infty$ and then letting
	$\epsilon\to0$ gives
	$	\overline h_{\scl,\mathsf S}^{C}
	\left(T^{(m)},X^{(m)}\right)\geq\gamma.$

	By Lemma~\ref{lem 5.3}, there exists an equivariant embedding
	$\Phi_m:X^{(m)}\to\mathscr M(X)$.
	Uniform continuity of its inverse on the compact image
	implies that, for every $\eta>0$, there exists $\epsilon>0$
	such that
	\begin{align*}
		D\left(\Phi_m(x),\Phi_m(y)\right)<\epsilon
		\quad\Longrightarrow\quad
		d^{(m)}(x,y)<\eta.
	\end{align*}
	By the equivariance, we obtain
	$s_n^{\mathsf S}\left(X^{(m)},d^{(m)},\eta\right)
	\leq
	s_n^{\mathsf S}\left(\mathscr M(X),D,\epsilon\right).$
	It follows that
	\begin{align*}
		\overline h_{\scl,\mathsf S}^{C}
		\left(T_*,\mathscr M(X)\right)
		\geq
		\overline h_{\scl,\mathsf S}^{C}
		\left(T^{(m)},X^{(m)}\right)
		\geq\gamma.
	\end{align*}
	Since $\gamma>0$ was arbitrary, we have
	\begin{align*}
		\overline h_{\scl,\mathsf S}^{C}
		\left(T_*,\mathscr M(X)\right)=+\infty.
	\end{align*}
\end{proof}

	For an increasing sequence $\mathsf S=\{t_1<t_2<\cdots\}$, we can also consider the corresponding dimension. Define the upper dimension by
	\begin{align*}
		\overline D(\mathsf S)
		:=
		\inf\left\{
		\tau\geq0:
		\limsup_{n\to\infty}\frac{n}{t_n^\tau}=0
		\right\}.
	\end{align*}
	Denote $\mathcal P_{\scl}(X,T) := \left\{ \mathsf S: \overline h_{\scl,\mathsf S}^{C}(T,X)>0 \right\}$. Then
	the \emph{upper  sequence dimension scale   } is defined by
	\begin{align*}
		\overline D_{\scl}^{\mathrm{seq}}(X,T)
		:=
		\begin{cases}
			\displaystyle
			\sup_{\mathsf S\in\mathcal P_{\scl}(X,T)}
			\overline D(\mathsf S),
			&\mathcal P_{\scl}(X,T)\neq\emptyset,\\begin{align*}2mm]
			0,
			&\mathcal P_{\scl}(X,T)=\emptyset.
		\end{cases}
	\end{align*}

	\begin{corollary}
		\label{cor 5.12}
	If $\scl$ satisfies Condition~\ref{cond 2}, then
		\begin{align*}
			\overline D_{\scl}^{\mathrm{seq}}(X,T)
			=
			\overline D_{\scl}^{\mathrm{seq}}
			\left(\mathscr M(X),T_*\right).
		\end{align*}
	\end{corollary}

\begin{proof}
	By Theorem~\ref{thm 1.4}, for every increasing sequence
	$\mathsf S$,
	\begin{align*}
	\overline h_{\scl,\mathsf S}^{C}(T,X)>0
	\quad\Longleftrightarrow\quad
	\overline h_{\scl,\mathsf S}^{C}
	\left(T_*,\mathscr M(X)\right)>0.
	\end{align*}
	Hence
	$\mathcal P_{\scl}(X,T)
	=
	\mathcal P_{\scl}\left(\mathscr M(X),T_*\right)$
	and the conclusion follows from the definition.
\end{proof}

\section{An application in zero-entropy systems by entropy scales}\label{sec 6}
In this section, we use entropy scale to carry out a more refined classification and investigation of a class of zero-entropy systems.

	Topological entropy assigns the same value to all systems with
	subexponential orbit growth, making zero-entropy systems difficult to
	distinguish. Slow entropy provides a finer classification, but polynomial
	slow entropy remains infinite for many systems of intermediate complexity,
	including systems with stretched-exponential orbit growth. Entropy scales
	use more flexible gauges and can further distinguish such systems. The
	following construction shows that entropy scales can distinguish
	zero-entropy systems with the same entropy dimension but different finer
	orbit-growth rates.

	For a compact dynamical system $(X,T)$, we define the polynomial entropy (or slow entropy) and
	upper entropy dimension by
	\begin{align*}
		h_{\mathrm{pol}}(T,X)
		&=
		\lim_{\epsilon\to0}
		\limsup_{n\to\infty}
		\frac{\log r_n(X,\epsilon)}{\log n},\\
		\overline D_{\mathrm{ent}}(T,X)
		&=
		\inf\left\{
		s>0:
		\lim_{\epsilon\to0}
		\limsup_{n\to\infty}
		\frac{\log r_n(X,\epsilon)}{n^s}
		=0
		\right\}.
	\end{align*}
	It is easy to prove that
	$h_{\mathrm{pol}}(T,X)=\overline h_{\scl}^{C}(T,X)$ for the scaling
	$\scl_\alpha(e^{-n})=n^{-\alpha}$, and
	$\overline D_{\mathrm{ent}}(T,X)=\overline h_{\scl}^{C}(T,X)$ for the
	scaling $\scl_\alpha(e^{-n})=e^{-n^\alpha}$.

	We first show the scaling we considered and the construction of dynamical systems.
	\begin{itemize}
		\item \emph{The scaling}: Fix $0<\theta<1$ and define
		\begin{align*}
			\scl_\alpha(t)
			=
			e^{-\left(\log\frac1t\right)^\theta
				\left(
				\log\left(1+\log\frac1t\right)
				\right)^\alpha}.
		\end{align*}

		\item \emph{The construction of symbolic systems}: Fix $\rho>0$ and
		choose $j_0$ sufficiently large so that the sequence
		$\{a_j\}_{j\geq j_0}$ defined by
		\begin{align*}
			a_j
			=
			\left\lceil
			j^{\frac1\theta}
			\left(\log(j+1)\right)^{-\frac{\rho}{\theta}}
			\right\rceil
		\end{align*}
		is strictly increasing. Set
		$A=\{a_j:j\geq j_0\}\subset\mathbb Z_+$ and define
		$x_A
		:=
		\left(\mathbf 1_A(n)\right)_{n\in\mathbb Z}
		\in
		\{0,1\}^{\mathbb Z}.$
		Let $\sigma$ be the shift on $\{0,1\}^{\mathbb Z}$ and set
		\begin{align*}
			X_A
			:=
			\overline{
				\left\{
				x\in\{0,1\}^{\mathbb Z}:
				x\leq\sigma^k x_A
				\text{ for some }k\in\mathbb Z
				\right\}
			},
		\end{align*}
		where the inequality is understood coordinatewise. Then $X_A$ is a compact
		shift-invariant set.
	\end{itemize}

	\begin{proposition}
		\label{prop logarithmic correction}
		For the above scaling, we have	 $\overline h_{\scl}^{C}(\sigma,X_A)=\rho$ while
		\begin{align*}
			h_{top}(\sigma,X_A)=0,\quad
			h_{\mathrm{pol}}(\sigma,X_A)=+\infty,
			\quad 	\overline D_{\mathrm{ent}}(\sigma,X_A)=\theta.
		\end{align*}
		Consequently, for fixed $\theta$, the entropy scale distinguishes
		the systems corresponding to different values of $\rho$.
	\end{proposition}

	\begin{proof}
		For $n\in\mathbb N$, denote by
		\begin{align*}
			\mathcal L_n(X_A)
			=
			\left\{
			(x_0,\ldots,x_{n-1}):x\in X_A
			\right\}
		\end{align*}
		the set of admissible words of length $n$. Set $	p_{X_A}(n)=\#\mathcal L_n(X_A).$
		Define
		\begin{align*}
			F:[j_0,\infty)
			&\longrightarrow
			[F(j_0),\infty),\\
			F(x)
			&=
			x^{1/\theta}
			\left(\log(x+1)\right)^{-\rho/\theta}.
		\end{align*}
		By increasing $j_0$ if necessary, we may assume that $F$ is strictly
		increasing and convex on $[j_0,\infty)$.

		Since $F^{-1}(n)
		\asymp
		n^\theta\left(\log(n+1)\right)^\rho,$
		one can show
		\begin{align*}
			\#\left(A\cap[0,n)\right)
			=
			\#\left\{j\geq j_0:F(j)\leq n-1\right\}
			&=
			\left\lfloor F^{-1}(n-1)\right\rfloor-j_0+1\\
			&\asymp
			n^\theta\left(\log(n+1)\right)^\rho.
		\end{align*}
		Moreover, there exists a constant $C>0$ such that, for all
		sufficiently large $n$,
		\begin{align}
			\label{equ uniform counting}
			\sup_{u\in\mathbb Z}
			\#\left(A\cap[u,u+n)\right)
			\leq
			Cn^\theta\left(\log(n+1)\right)^\rho.
		\end{align}
		Indeed, if $A\cap[u,u+n)$ is nonempty, write $	A\cap[u,u+n)
		=
		\left\{a_r,\ldots,a_{r+q}\right\}.$
		Since $F$ is convex and $r\geq j_0$, we have
		\begin{align*}
			F(j_0+q)-F(j_0)
			\leq
			F(r+q)-F(r)\leq n+2.
		\end{align*}
		Consequently, $j_0+q
		\leq
		F^{-1}\left(n+2+F(j_0)\right).$
		Then for some $C>1$,
		\begin{align*}
			q+1
			\leq
			Cn^\theta\left(\log(n+1)\right)^\rho,
		\end{align*}
		where $C$ is independent of $u$ and $n$. Then we prove
		\eqref{equ uniform counting}.

		\begin{claim}
			There exists $C_0>0$ such that, for all sufficiently large $n$,
			\begin{align*}
				\#\left\{
				(A-u)\cap[0,n):u\in\mathbb Z
				\right\}
				\leq n^{C_0}.
			\end{align*}
		\end{claim}

		\begin{proof}
			Set $\Delta_j:=a_{j+1}-a_j$. For all sufficiently large $j$,
			\begin{align*}
				\Delta_j
				\asymp
				j^{\frac{1-\theta}{\theta}}
				\left(\log(j+1)\right)^{-\frac{\rho}{\theta}}.
			\end{align*}
			Choose
			$0<\kappa<\frac{1-\theta}{\theta}$.
			Then $\Delta_j\geq j^\kappa$ for all sufficiently large $j$.
			Therefore,
			\begin{align*}
				\Delta_j\leq n
				\quad\Longrightarrow\quad
				j\leq n^{\frac{1}{\kappa}},
			\end{align*}
			apart from finitely many indices. Hence, for some constant
			$C_1>0$,
			\begin{align*}
				\#\left\{j\geq j_0:\Delta_j\leq n\right\}
				\leq
				C_1n^{\frac{1}{\kappa}}.
			\end{align*}

			If $[u,u+n)$ contains at least two points of $A$, then it contains
			two consecutive points $a_j$ and $a_{j+1}$, and necessarily
			$\Delta_j<n$. For a fixed such $j$, the condition
			$a_j,a_{j+1}\in[u,u+n)$ implies
			$	a_{j+1}-n<u\leq a_j,$
			so there are at most $n+1$ possible integer values of $u$.
			Thus the patterns containing at least two points number at most
			$C_1(n+1)n^{\frac{1}{\kappa}}$.
			Patterns containing at most one point have only $n+1$
			possibilities. Consequently,
			\begin{align*}
				\#\left\{
				(A-u)\cap[0,n):u\in\mathbb Z
				\right\}
				\leq
				C_1(n+1)n^{\frac{1}{\kappa}}+n+1
				\leq
				n^{C_0}
			\end{align*}
			for all sufficiently large $n$, where we may take $C_0=2+\frac{1}{\kappa}.$
		\end{proof}

		By the definition of $X_A$, every subset of $A\cap[0,n)$ gives a
		word in $\mathcal L_n(X_A)$ whose entries are $1$ exactly at the
		points of that subset. Conversely, for every word
		$w\in\mathcal L_n(X_A)$, the positions at which $w$ equals $1$ are
		contained in $(A-u)\cap[0,n)$ for some $u\in\mathbb Z$. Hence, by the
		claim above and \eqref{equ uniform counting}, one has
		\begin{align*}
			2^{\#\left(A\cap[0,n)\right)}
			\leq
			p_{X_A}(n)
			\leq
			n^{C_0}
			2^{Cn^\theta\left(\log(n+1)\right)^\rho}.
		\end{align*}
		Together with the estimate for
		$\#\left(A\cap[0,n)\right)$, this gives $\log p_{X_A}(n)
		\asymp
		n^\theta\left(\log(n+1)\right)^\rho.$ For
		every sufficiently small $\epsilon>0$, there exists
		$q_\epsilon\in\mathbb N$ such that
		\begin{align*}
			p_{X_A}(n)
			\leq
			r_n(X_A,\epsilon)
			\leq
			p_{X_A}(n+2q_\epsilon).
		\end{align*}
		Therefore, $	\log r_n(X_A,\epsilon)
		\asymp
		n^\theta\left(\log(n+1)\right)^\rho.$
		It shows $\overline h_{\scl}^{C}(\sigma,X_A)=\rho.$
		Moreover,
		\begin{align*}	h_{top}(\sigma,X_A)=0,\quad
			h_{\mathrm{pol}}(\sigma,X_A)=+\infty, \quad
			\overline D_{\mathrm{ent}}(\sigma,X_A)=\theta.
		\end{align*}
		This completes the proof.
	\end{proof}

	\section*{Acknowledgments}
	\noindent The authors thank Professor Stefano Galatolo for his helpful comments and suggestions.
	This work was supported by the National Natural Science Foundation of China (No.~12471184). The
	third author was also supported by the Qinglan Project of Jiangsu Province, China.

	\section*{Statements and Declarations}
	\subsection*{Data Availability}	No datasets were generated or analysed during the current study.
	\subsection*{Competing interests}The authors declare no competing interests.

\end{document}